\documentclass[12pt,a4paper]{amsart}

\usepackage[utf8]{inputenc}

\usepackage{mathtools}
\usepackage{amsfonts}
\usepackage{amssymb}
\usepackage{amsthm}

\usepackage[usenames,dvipsnames,svgnames,table]{xcolor}
\usepackage{enumerate}
\usepackage{ytableau}
\usepackage[makeroom]{cancel}
\usepackage{paralist}
\usepackage{hyperref}
\usepackage{caption}

\DeclarePairedDelimiter{\abs}{\lvert}{\rvert}
\newcommand*{\aA}{\mathcal{A}}
\newcommand*{\bB}{\mathcal{B}}
\newcommand*{\cC}{\mathfrak{C}}
\newcommand*{\pP}{{\smash{\mathrm{P}}}}
\newcommand*{\pPl}{{\smash{\mathrm{P}}}_{\smash{\mathrm{lps}}}}
\newcommand*{\pPr}{{\smash{\mathrm{P}}}_{\smash{\mathrm{rps}}}}
\newcommand*{\qQ}{{\smash{\mathrm{Q}}}}
\newcommand*{\qQl}{{\smash{\mathrm{Q}}}_{\smash{\mathrm{lps}}}}
\newcommand*{\qQr}{{\smash{\mathrm{Q}}}_{\smash{\mathrm{rps}}}}
\newcommand*{\bell}{{\smash{\mathrm{lps}}}}
\newcommand*{\belr}{{\smash{\mathrm{rps}}}}
\newcommand*{\xps}{{\smash{\mathrm{xps}}}}
\newcommand*{\smx}{{\smash{\mathrm{x}}}}
\newcommand*{\dstd}{{\smash{\mathrm{std}}}^{-1}}
\newcommand*{\stdl}{{\smash{\mathrm{std}}}_{\smash{\mathrm{l}}}}
\newcommand*{\stdr}{{\smash{\mathrm{std}}}_{\smash{\mathrm{r}}}}
\newcommand*{\std}{{\smash{\mathrm{std}}}}
\newcommand*{\Std}{{\smash{\mathrm{Std}}}}

\newcommand*{\Dstd}{{\smash{\mathrm{Dstd}}}}
\newcommand*{\cont}{{\smash{\mathrm{cont}}}}
\newcommand*{\sh}{{\smash{\mathrm{sh}}}}
\newcommand*{\tab}{{\smash{\mathrm{Tab}}}}
\newcommand*{\pstab}{{\smash{\mathrm{PSTab}}}}
\newcommand*{\sml}{\smash{\mathrm{l}}}
\newcommand*{\smr}{\smash{\mathrm{r}}}

\makeatletter
\newcommand{\biggg}{\bBigg@{4}}
\newcommand{\bigggl}{\mathopen\biggg}

\newcommand{\bigggr}{\mathclose\biggg}
\makeatother

\newcommand*\lps{lPS}
\newcommand*\rps{rPS}
\newcommand*\ps{\text{PS}}

\DeclarePairedDelimiter{\parens}{\lparen}{\rparen}
\newcommand*{\evlit}{{\mathrm{ev}}}
\newcommand*{\ev}[2][]{\evlit\parens[#1]{#2}}

\theoremstyle{plain}
\newtheorem{thm}{Theorem}[section]
\newtheorem{lem}[thm]{Lemma}
\newtheorem{prop}[thm]{Proposition}
\newtheorem{cor}[thm]{Corollary}

\newtheorem{remark}[thm]{Remark}

\theoremstyle{definition}

\newtheorem{algorithm}[thm]{Algorithm}

\makeindex

\title{Combinatorics of patience sorting monoids}

\author{Alan J. Cain}
\address{Centro de Matem\'atica e Aplica\c{c}\~oes, Faculdade de Ci\^encias e
Tecnologia, Universidade Nova de Lisboa,
	2829-516 Caparica, Portugal}
\email{a.cain@fct.unl.pt}
\thanks{The first author was supported by
an Investigador FCT fellowship (IF/01622/2013/CP1161/CT0001).}

\author{Ant\'onio
Malheiro}
\address{Centro de Matem\'atica e Aplica\c{c}\~oes and Departamento de
	Matem\'atica, Faculdade de Ci\^encias e Tecnologia, Universidade Nova de
Lisboa,
	2829-516 Caparica, Portugal}
\email{ajm@fct.unl.pt}
\thanks{For the first two authors,
this work was partially supported by the Funda\c{c}\~ao para a
Ci\^encia e a Tecnologia (Portuguese Foundation for Science and Technology)
through the project UID/MAT/00297/2013 (Centro de Matem\'atica e
Aplica\c{c}\~oes), and the project PTDC/MHC-FIL/2583/2014}

\author{F\'abio M. Silva}
\address{Departamento de Matem\'atica
	and CEMAT-CI\^ENCIAS, Faculdade de Ci\^encias, Universidade de Lisboa,
Lisboa
	1749-016, Portugal.}
\email{femsilva@fc.ul.pt}
\thanks{The third author was supported by an FCT Lismath fellowship
(PD/BD/52644/2014) and partially supported by the FCT project CEMAT-Ci\^encias
UID/Multi/04621/2013}

\begin{document}

\begin{abstract}
  This paper makes a combinatorial study of the two monoids and the two types of tableaux that arise from the two
  possible generalizations of the Patience Sorting algorithm from permutations (or standard words) to words.  For both
  types of tableaux, we present Robinson--Schensted--Knuth-type correspondences (that is, bijective correspondences
  between word arrays and certain pairs of semistandard tableaux of the same shape), generalizing two known
  correspondences: a bijective correspondence between standard words and certain pairs of standard tableaux, and an
  injective correspondence between words and pairs of tableaux.

  We also exhibit formulas to count both the number of each type of tableaux with given evaluations (that is, containing
  a given number of each symbol). Observing that for any natural number $n$, the $n$-th Bell number is given by the
  number of standard tableaux containing $n$ symbols, we restrict the previous formulas to standard words and extract a
  formula for the Bell numbers. Finally, we present a `hook length formula' that gives the number of standard tableaux
  of a given shape and deduce some consequences.
\end{abstract}


\maketitle
\tableofcontents

\section{Introduction}

Patience Sorting is a one-person card game that was created as a method to sort decks of cards. The name has its origins
in the works of Mallows, who credited the game to Ross \cite{Mallows62,10.2307/2028347}.

The idea of the game is to split a given shuffled deck of cards on the
symbols $1,2,\ldots,n$ into sorted subsequences called columns (piles
according to \cite{MR1694204,BL2005}) using Mallows' Patience Sorting
procedure, with the goal of finishing with as few piles as possible
\cite{MR1694204}. The decks of cards can therefore be seen as standard
words from the symmetric group of order $n$ (denoted $\mathfrak{S}_n$). According
to this notation, following Algorithm~1.1 of \cite{BL2005}, this
procedure can be described in the following way
\begin{algorithm}[Mallows' Patience Sorting procedure]
	\label{alg:MallowPS}
	~\par\nobreak
	\textit{Input:} A standard word $\sigma=\sigma_1\sigma_2\cdots
	\sigma_n \in \mathfrak{S}_n$.

	\textit{Output:} A set of columns $\{c_1,c_2,\ldots,c_m\}$.

	\textit{Method:}
	\begin{itemize}
		\item Put $\sigma_1$ on the bottom of the first column, $c_1$;

		\item for each remaining symbol $\sigma_i$, with $i\in\{2,\ldots,
		n\}$, if $b_1,...,b_k$ are the symbols on the bottom of the
		columns $c_1,...,c_k$ that have already been formed:
		\begin{itemize}
			\item if $\sigma_i>b_k$, then put the symbol $\sigma_i$ into
			the bottom of the column $c_{k+1}$;
			\item otherwise, find the leftmost symbol $b_j$ from $b_1,
			\ldots, b_k$ that is greater than $\sigma_i$ and put $\sigma_i$
			on the bottom of that column.
		\end{itemize}
	\end{itemize}
\end{algorithm}

As noted in \cite{BL2005,MR1694204}, there are certain similarities between
this algorithm and Schensted's insertion algorithm for standard words. In
fact, as noted by the authors of \cite{BL2005}, this algorithm can be seen
as a non-recursive version of Schensted's insertion algorithm. This
perspective suggests that there would be an analogue of the Robinson correspondence (cf.
\cite{MR1464693}): a bijection between
standard words and pairs of standard Young tableaux of the same
shape, discovered by Robinson in \cite{10.2307/2371609} and
later rediscovered by Schensted for the context of words \cite{MR0121305}.
Burstein and Lankham \cite{BL2005} constructed such an analogue by considering an algorithm that simultaneously
performs Mallows' Patience Sorting procedure and the recording of that procedure.

Following the terminological distinction made in \cite{MR1464693}, the original Robinson correspondence (between
standard words and pairs of standard Young tableaux of the same shape) has two extensions:
\begin{itemize}
	\item the Robinson--Schensted correspondence, which is a
	bijection between words and pairs $(P,Q)$, where $P$ is a semistandard 
Young tableau and $Q$ is a standard Young tableaux of the same shape; and
	\item  the Robinson--Schensted--Knuth correspondence
	\cite{knuth1970permutations}, a bijection between two-rowed lexicographic
	arrays and pairs of semistandard Young tableaux having the
	same shape.
\end{itemize}
Moreover, Schensted's insertion algorithm gives rise to a combinatorial monoid, the plactic monoid \cite{MR646486},
which can be described as the quotient of the free monoid over the congruence which relates words in that free monoid
that yield the same semistandard Young tableau under Schensted's insertion algorithm (see \cite[\S~6]{MR1905123} for more details).

Regarding the combinatorial monoid part, considering a possible
ge\-ne\-ra\-li\-za\-tion of Algorithm~\ref{alg:MallowPS} to words and using a
similar method, Rey constructed the Bell monoid \cite{Maxime07}. In fact,
in that same paper, Rey proposes a Robinson--Schensted-like correspondence
between words and pairs of \ps\ tableaux (Bell tableaux,
in the notation of \cite{Maxime07}) $(P,Q)$ having the same shape, $P$
being semistandard and $Q$ standard. However, this correspondence proves
to be only an injection.

In \cite{MR1694204} two possible generalizations of Patience Sorting
for words were proposed, one of them coinciding with the one studied
by Rey. So, following these generalizations, in \cite{1706.06884} the present authors
constructed and studied two distinct monoids, the $\bell$ monoid (which
coincides with the Bell monoid of \cite{Maxime07}) arising from \lps\
tableaux and the $\belr$ monoid arising from \rps\ tableaux.

Our goal in Section~3 is to provide bijective Robinson--Schensted-like
and Robinson--Schensted--Knuth-like correspondences for both the $\bell$
and the $\belr$ monoids, extending the ideas of \cite{BL2005}.

In sections 4 and 5 we turn to another topic, namely counting \lps\
and \rps\ tableaux. As noted in \cite{BL2005}, given a totally
ordered alphabet $\bB$ and a partition of $\bB$, there is a natural
identification between the sets that compose that partition of $\bB$
and the columns of a standard \ps\ tableau obtained from the
insertion of certain standard words over $\bB$ under
Algorithm~\ref{alg:MallowPS}. This identification gives rise to a
one-to-one correspondence between the partitions of $\bB$ and the
standard \ps\ tableaux obtained from applying
Algorithm~\ref{alg:MallowPS} to the standard words over $\bB$. Since
the $n$-th Bell number counts the number of distinct ways to partition
a set with $n$ distinct symbols (cf. sequence A000110 in the OEIS), the
$n$-th Bell number also provides the number of standard \ps\ tableaux
over a totally ordered alphabet with $n$ distinct symbols. Furthermore,
as the Stirling number of second kind, $S(n,k)$, is equal to the number
of distinct ways to partition $n$ distinct elements into $k$ distinct
sets \cite{MR1411676}, the sum of $S(n,k)$ with $k$ ranging
between $1$ and $n$ is also equal to $n$-th Bell number. In Section~4
we will follow a similar strategy. More specifically, we will count
the number of \lps\ (resp. \rps) tableaux over words with the
same number of each symbols (that is, with the same
evaluation) and having the same \lps\ (\rps) bottom row. Then, by summing
over all the possible \lps\ (\rps) bottom rows, we extract a formula to
count the number of \lps\ (\rps) tableaux over words with the
same evaluation. Restricting the \lps\ and \rps\ formulas to standard
\ps\ tableaux we will see that they coincide and we will deduce
a formula to count the Bell numbers.

Section~5 also gives a formula connecting the \ps\ tableaux and Bell numbers. However,
we will follow an entirely different approach. Our strategy will be to provide an
analogue of the `hook length formula', which gives the number
of standard Young tableaux over $\{1,\ldots,n\}$ with a given
shape.
This rule is due to J.S. Frame, G. de B. Robinson, and R.M. Thrall \cite{frame1954hook} and has been generalized to
other combinatorial objects such as shifted Young tableaux (see \cite{thrall1952combinatorial},
\cite{sagan1980selecting}) and trees (see \cite{sagan1989probabilistic}). In this paper we provide an analogue of this
result for standard \ps\ tableaux over an arbitrary totally ordered alphabet $\bB$. By summing the hook length formulas
over all the possible shapes for standard \ps\ tableaux over $\bB$, we also obtain the Bell number of order $\abs{\bB}$.
Furthermore, this rule together with the injectivity of the Robinson--Schensted-like correspondence from \cite{Maxime07}
allows us to deduce an upper bound on the number of standard words that insert to a specific standard \ps\ tableau.

%
%
%
%
%
%
%
%
%

\section{Preliminaries and notation}

In this section we introduce the notions that we shall use along the paper.
For more details concerning these constructions see for instance \cite{MR1905123},
\cite{MR2142078}, \cite{howie1995fundamentals} and
\cite{1706.06884}.
%

\subsection{Words and standardization}
\label{alphabetswords}

The alphabets that we will consider on this paper will always be totally ordered.

Given an alphabet $\bB$, $\bB^*$ denotes the free monoid
over $\bB$. In this paper, $\aA=\{1 < 2 < 3 <\ldots \}$ denotes the
set of natural numbers viewed as an infinite totally ordered alphabet.
Furthermore, for any $n\in \mathbb{N}$, the set $\aA_n$ denotes the totally
ordered subset of $\aA$, on the symbols $1,\ldots, n$.


A \emph{word} over an arbitrary alphabet is an element of the free
monoid over that alphabet, with the symbol $\varepsilon$ denoting the
empty word. In particular, a \emph{standard word} over a finite (resp. infinite)
alphabet is a word where  each symbol from the alphabet occurs exactly (at most) once.
For any finite  alphabet $\bB$, denote by $\mathfrak{S}(
\bB)$ the set of standard words over $\bB$. For example, if
$\bB=\{2,4,5\}$, then
\begin{align*}
\mathfrak{S}(\bB)=\{245,254,425,452,524,542\}.
\end{align*}
In the following paragraphs we will define several concepts that are
directly related with the notion of word. So, if $\bB$ is an
arbitrary alphabet and $w=w_1\cdots w_m\in\bB^*$, with
$w_1,\ldots,w_m\in \bB$,
then:
\begin{itemize}
	\item the \emph{length of $w$}, $\abs{w}$, is the number of symbols
	from $\bB$ that compose $w$, counting repetitions. If $w=\varepsilon$,
	then $\abs{w}=0$;
	\item 
	a word $u\in \bB^*$ is said to be a \emph{subword of} $w$ if there 
exists
	a sequence of indexes, $i_1,\ldots, i_k\in \mathbb{N}$, with $1\leq i_1
	<\ldots<i_k\leq m$, such that $u={w}_{{i}_1}\cdots {w}_{i_k}$;
	\item for any word $u\in \bB^*$, $u$ is a \emph{factor} of
	$w$ if there exist words $v_1,v_2\in \bB^*$,
	such that $w=v_1uv_2$;
	\item for any $a\in \bB$, the number of occurrences of $a$ in $w$, is
	denoted by $\abs{w}_a$;
	\item the \emph{content of} $w$, is the set $\cont(w)=\left\{a\in\bB:
	\abs{w}_a\geq 1\right\}$;
	\item the \emph{evaluation of} a word $w\in \bB^*$, denoted $\ev{w}$, 
is the $|\bB|$-tuple of non-negative integers, indexed in increasing order by 
the elements of $\bB$, whose $a$-th term is $\abs{w}_a$.
\end{itemize}

Next, we define two different processes for standardizing words over $\bB$,
both allowing us to assign to any word from $\bB^*$, a standard word of the
same length over a new alphabet.

Consider the alphabet
\begin{equation*}
\mathcal{C}(\bB)=\{a_{i}: a\in \bB\ \wedge\ i\in \aA\}
\end{equation*}
totally ordered in the following way: for any $a_i,c_j\in\mathcal{C}(\bB)$
\begin{equation*}
a_{i}<_{\mathcal{C}(\bB)} c_{j} \Leftrightarrow a<_{\bB}c \vee \left(a=c\wedge i<_{\aA}j\right).
\end{equation*}

Given a word $w$ over $\bB$,
$\stdl(w)$ denotes the \emph{left to right standardization of $w$}, which is
the word obtained by reading $w$ from left to right and attaching
to each symbol $a\in \cont(w)$ an index $i$ to the $i$-th occurrence of $a$.
The \emph{right to left standardization of $w$}, $\stdr(w)$, is obtained in
the same way, but reading the word from right to left. For example, considering
the word $w=4124321
\in\aA_4^*$,
\begin{alignat*}{9}
&w\quad && =\quad &&{\color{Green}4}\ &&1\ &&{\color{red}2}\ &&{\color{Green}4}\
&&{\color{blue}3}\ &&{\color{red}2}\ &&1\ \\
&\stdl(w)\quad && =\quad  &&{\color{Green}4_1}\ &&1_1\ &&{\color{red}2_1}\
&&{\color{Green}4_2}\ &&{\color{blue}3_1}\ &&{\color{red}2_2}\ &&1_2\ \\
&\stdr(w)\quad && =\quad  &&{\color{Green}4_2}\ &&1_2\ &&{\color{red}2_2}\
&&{\color{Green}4_1}\ &&{\color{blue}3_1}\ &&{\color{red}2_1}\ &&1_1\
\end{alignat*}
For any word $w$ over $\bB$, it is straightforward that both $\stdl(w)$ and $\stdr(w)$
are standard words over the alphabets $\cont(\stdl(w))\subseteq \mathcal{C}(\bB)$
and $\cont(\stdr(w))\subseteq \mathcal{C}(\bB)$, respectively, with the order
induced from $\mathcal{C}(\bB)$.

Henceforth, any word arising from the processes of left to right or right to
left standardization will be called a \emph{standardized word}. Given any
symbol $a_{b}\in\mathcal{C}(\bB)$, the \emph{underlying symbol of}
$a_{b}$ is $a$ and the \emph{index of} $a_{b}$ is $b$.


Given any standardized word $w\in\mathcal{C}(\bB)$, the \emph{de-standardization 
of} $w$,
$\dstd(w)\in\bB$, is the word obtained from $w$ by erasing all of its indexes.
  
By the uniqueness of the writing of words, together with the fact that any
symbol in any position of a given word coincides with the underlying symbol
in the same position of the word obtained through any of the previous
processes of standardization, we have the following


\begin{remark}\label{rmk:bijection_words}
	For each word $w$, using processes of standardization previously
	described we obtain unique standardized words $\stdl(w)$ and $\stdr(w)$.
	Moreover, by applying de-standardization to any of these standardized words,
	we obtain the word $w$.
\end{remark}

\subsection{\ps\ tableaux and insertion}

In this subsection we introduce the concept of pre-tableaux and recall
the basic concepts regarding patience sorting (\ps) tableaux over an
arbitrary totally ordered alphabet $\bB$. We will also recall the insertion
on \ps\ tableaux.


A \emph{composition diagram} is a concatenation of a finite collection of boxes 
arranged in  bottom-justified columns. Such a diagram is said to be a 
\emph{column diagram} if all its boxes are disposed vertically. 

A \emph{pre-column} over $\bB$ is a column diagram where each of 
its
boxes is filled with a symbol from $\bB$ and such that distinct boxes have
distinct symbols.
A \emph{pre-tableau over} $\bB$ is a composition diagram where all boxes 
are filled with symbols from $\bB$
 in such a way that distinct boxes have distinct symbols. For
instance, if
\begin{equation*} \label{exmp1}
\ytableausetup
{aligntableaux=center}
C=\begin{ytableau}
\none &  & \none & \none \\
\none &  &  &  \\
&  &  &
\end{ytableau},\ \text{ and }\
P=\begin{ytableau}
\none & 5 & \none \\
2 & 1 & 4 \\
7 & 3 & 6
\end{ytableau},
\end{equation*}
then $C$ is a composition diagram and $P$ is a pre-tableau.


An \emph{\lps\ (\rps) tableau} over $\bB$ is
a composition diagram where each box is filled with a symbol from $\bB$ and
such that when reading the symbols
of its columns from top to bottom they are in strictly (weakly) 
decreasing order and when reading the symbols in the bottom-row from left to 
right, they are in weakly
(strictly) increasing order.

A
\emph{standard \ps\ tableau over $\bB$} is an \lps\ tableau (or
alternatively an \rps\ tableau) such that each symbol of $\bB$ occurs at
most once, whereas a \emph{recording tableau} is a standard \ps\ tableau such
that if $m$ is the number of boxes in its underlying composition diagram, then
each symbol from $\aA_{m}$ occurs exactly once. For instance, if
\begin{align*} \label{exmp2}
\ytableausetup
{aligntableaux=center}
Q=\begin{ytableau}
\none &\none  & 4 \\
4 & 2 & 3 \\
1 & 1 & 2
\end{ytableau},\ R=\begin{ytableau}
4 &\none  & \none \\
1 & 2 & 4 \\
1 & 2 & 3
\end{ytableau},\ \text{ and }\
S=\begin{ytableau}
\none &\none  & 7 \\
6 & 3 & 5 \\
1 & 2 & 4
\end{ytableau},
\end{align*}
then $Q$ is an \lps\ tableau, $R$ is
an \rps\ tableau and $S$ is both a standard \ps\ and a recording
tableau. The tableaux $Q$ and $R$ can be considered as
tableaux over the alphabets $\aA$ or $\aA_n$ (for $n\geq 4$), whereas the
tableaux $P$ and $S$ can be viewed as tableaux over $\aA$ or $\aA_n$
(for $n\geq 7$).

For any natural $n\in\mathbb{N}$, $\lambda=(\lambda_1,\ldots,\lambda_m)\in
\mathbb{N}^m$ is a \emph{composition of} $n$, denoted by $\lambda\vDash n$, if
$\lambda_1 +\cdots+ \lambda_m=n$. Given an arbitrary pre-tableau $T$, such
that $c_1,c_2,\ldots,c_m$ are its columns from left to right, we will often refer
to this tableau using the notation $c_1c_2\cdots c_m$ instead of $T$. We make the following definitions:
\begin{itemize}
	\item the \emph{content of} $T$, $\cont(T)$, is the set composed by the symbols
	that occur in $T$;
	\item for any $i\in\{1,\ldots, m\}$, the \emph{length of the column} $c_i$
	is the number of boxes that compose $c_i$ and is denoted by $\abs{c_i}$, the
	\emph{length of the bottom row of} $T$ is equal to $m$, and the \emph{length
	of} $T$, $\abs{T}$, is given by the sum of the lengths of its columns
	$\abs{T}=\abs{c_1}+\cdots+\abs{c_m}$;
	\item the \emph{shape of} $T$, denoted by $\sh(T)$ is the composition formed
	by the lengths of the columns of $T$ from left to right, that is, $\sh(T)=\left(
	\abs{c_1},\ldots,\abs{c_m}\right)$. Note that this notion of shape is the dual of
	the usual notion of shape, as given in \cite{MR1905123};
	\item if $a$ belongs to the content of $T$ and $a$ occurs in the $i$-th
	column of $T$, counting columns from left to right and in the $j$-th
	box of the $i$-th column counting boxes from bottom to top, then the
	\emph{column-row position of} $a$ is given by the pair $(i,j)\in
	\mathbb{N}\times \mathbb{N}$;
	\item if $T$ is a tableaux over $\bB$, then for any $a\in
	\bB$, $\abs{T}_a$ denotes the number of symbols $a$ occurring in
	$T$ and the \emph{evaluation of} $T$ is the sequence
	(infinite or finite of length $\abs{\bB}$ if $\bB$ is,
	respectively, infinite or finite) whose $a$-th term is $\abs{T}_a$,
	for any $a\in \bB$.
\end{itemize}
Considering the standard \rps\ tableau $R$ given in the previous example,
the content of $R$ is given by the set $\cont(R)=\{1,2,3,4\}$, the length of $R$ is
$7$, the shape of $R$ is given by $\sh(R)=(3,2,2)$, the symbol $3$ occurs in the
column-row position	$(3,1)$ of $R$, seen as a tableau over $\aA$, the
evaluation of $R$ is given by the infinite sequence $(2,2,1,2,0,0,\ldots)$, seen as
a tableau over $\aA_4$, the evaluation of $R$ is given by the sequence $(2,
2,1,2)$.

Any two pre-tableaux are said to be equal if they have the same
shape and if in corresponding column-row positions of the tableaux
the symbols are equal. If $\bB\subseteq \aA_n$ for some $n$, and $\lambda\vDash\lvert{\bB}
\rvert$, then $\pstab_\bB(\mathcal{\lambda})$ denotes the set of standard \ps\
tableaux with shape $\lambda$ and whose content is $\bB$. Similarly,
$\tab_\bB(\lambda)$ denotes the set of pre-tableaux with shape $\lambda$
and whose content is $\bB$. Furthermore,
\begin{equation*}
\pstab_\bB=\bigcup_{\substack{\lambda\vDash \lvert{\bB}\rvert}} \pstab_\bB(\lambda)\ \text{ and }\
\tab_\bB=\bigcup_{\lambda\vDash \lvert{\bB}\rvert}
\tab_\bB(\lambda).
\end{equation*}
For example,
\begin{align*}
\ytableausetup
{mathmode, boxsize=1.1em, aligntableaux=center}
\pstab_{\{2,4,5\}}\left((2,1)\right)&=\left\{\begin{ytableau}
4\\
2 & 5
\end{ytableau},\begin{ytableau}
5\\
2 & 4
\end{ytableau}\right\}\\
\pstab_{\{2,4,5\}}&=\left\{\begin{ytableau}
5\\
4\\
2
\end{ytableau},\begin{ytableau}
4\\
2 & 5
\end{ytableau},\begin{ytableau}
5\\
2 & 4
\end{ytableau},\begin{ytableau}
\none & 5\\
2 & 4
\end{ytableau},\begin{ytableau}
2 & 4 & 5
\end{ytableau}\right\},
\end{align*}
while
\begin{align*}
\ytableausetup
{mathmode, boxsize=1.1em, aligntableaux=center}
&\null\kern-20mm\tab_{\{2,4,5\}}((2,1))\\
={}&\left\{\begin{ytableau}
2\\
4 & 5
\end{ytableau},\begin{ytableau}
2\\
5 & 4
\end{ytableau},\begin{ytableau}
4\\
2 & 5
\end{ytableau},\begin{ytableau}
4\\
5 & 2
\end{ytableau},\begin{ytableau}
5\\
2 & 4
\end{ytableau},\begin{ytableau}
5\\
4 & 2
\end{ytableau}\right\}
\displaybreak[0]\\
\tab_{\{2,4,5\}}={}&\bigggl\{\begin{ytableau}
2\\
4\\
5
\end{ytableau},\begin{ytableau}
2\\
5\\
4
\end{ytableau},\begin{ytableau}
4\\
2\\
5
\end{ytableau},\begin{ytableau}
4\\
5\\
2
\end{ytableau},\begin{ytableau}
5\\
2\\
4
\end{ytableau},\begin{ytableau}
5\\
4\\
2
\end{ytableau},\begin{ytableau}
2\\
4 & 5
\end{ytableau},\begin{ytableau}
2\\
5 & 4
\end{ytableau},\begin{ytableau}
4\\
2 & 5
\end{ytableau},\begin{ytableau}
4\\
5 & 2
\end{ytableau},\\
&\quad\begin{ytableau}
5\\
2 & 4
\end{ytableau},\begin{ytableau}
5\\
4 & 2
\end{ytableau},\begin{ytableau}
\none & 2\\
4 & 5
\end{ytableau},\begin{ytableau}
\none & 2\\
5 & 4
\end{ytableau},\begin{ytableau}
\none & 4\\
2 & 5
\end{ytableau},\begin{ytableau}
\none & 4\\
5 & 2
\end{ytableau},\begin{ytableau}
\none & 5\\
2 & 4
\end{ytableau},\begin{ytableau}
\none & 5\\
4 & 2
\end{ytableau},\\[-4mm]
&\quad\begin{ytableau}
2 & 4 & 5
\end{ytableau}, \begin{ytableau}
2 & 5 & 4
\end{ytableau},\begin{ytableau}
4 & 2 & 5
\end{ytableau},\begin{ytableau}
4 & 5 & 2
\end{ytableau},\begin{ytableau}
5 & 2 & 4
\end{ytableau},\begin{ytableau}
5 & 4 & 2
\end{ytableau}\bigggr\}.
\end{align*}


The following algorithm describes the insertion of an arbitrary word into
a \ps\ tableau, merging in one algorithm the algorithms 3.1 and 3.2
of \cite{1706.06884}. We will use the notation $\pPl()$, $\pPr()$
instead of the notation $\mathfrak{R}_\ell()$, $\mathfrak{R}_r()$, respectively, of
\cite{1706.06884}.

\begin{algorithm}[\ps\ insertion of a word]\label{alg:PSinsertion}
	~\par\nobreak
	\textit{Input:} A word $w$ over a totally ordered alphabet.

	\textit{Output:} An \lps\ tableau $\pPl(w)$ (resp., \rps\
	tableau $\pPr(w)$).

	\textit{Method:}
	\begin{enumerate}
		\item If $w=\varepsilon$, output the empty tableaux
		$\emptyset$. Otherwise:
		\item $w=w_1\cdots w_n$, with symbols $w_1,\ldots,w_n$ of the 
alphabet.
		Setting
		\begin{align*}
		\pPl(w_1)=\ytableausetup
		{boxsize=1.1em, aligntableaux=center}\begin{ytableau}
		w_1 \end{ytableau} =\pPr(w_1),
		\end{align*}
		then, for each remaining symbol $w_j$ with $1<j\leq n$,
		denoting by $r_1\leq \dots\leq r_k$ (resp., $r_1< \dots< r_k$) 
the
		symbols in the bottom row of the tableau $\pPl(w_1\cdots w_{j-1})$
		(resp., $\pPr(w_1\cdots w_{j-1})$):
		\begin{itemize}
			\item if $r_k\leq w_j$ (resp., $r_k < w_j$), 
insert $w_j$ in a new
			column to the right of $r_k$ in $\pPl(w_1\cdots w_{j-1})$ (resp., $\pPr(
			w_1\cdots w_{j-1})$);

                      \item otherwise, if $m=\min\left\{i\in\{1,\ldots, k\}: 
w_j < r_i \right\}$,
                        (resp. $m=\min\left\{i\in\{1,\ldots, k\}:w_j\leq 
r_i\right\}$) construct a new empty box on top
                        of the $m$-th column of $\pPl(w_1\cdots w_{j-1})$ (resp. $\pPr(w_1\cdots w_{j-1})$). Then bump
                        all the symbols of the column containing $r_m$ to the box above and insert $w_j$ in the box
                        which has been cleared and previously contained the symbol $r_m$.
		\end{itemize}
		Output the resulting tableau.
	\end{enumerate}
\end{algorithm}

For brevity, we will use $\pP()$, whenever no distinction between
$\pPl()$ and $\pPr()$ is needed. Given $\smx\in \{\sml,\smr\}$ and words $u,v$ of
the same length (possibly over different totally ordered alphabets), we say that
$u$ and $v$ have \emph{equivalent $\smx$\ps\ insertions} if $\sh(\pP_\xps(u))=\sh(\pP_\xps(v)
)$ and for every $i=1,\ldots, |u|$ the $i$-th symbol of $u$ and the $i$-th symbol
of $v$ are in the same column-row position of the respective tableaux
$\pP_\xps(u)$ and $\pP_\xps(v)$. 

The following lemma is a restatement of
\cite[Lemma~3.4]{1706.06884} which relates the insertion of an arbitrary word with
its $x$-standardization.

\begin{lem}\label{lemma:equivalent_insertions}
	For any word $w$ over an alphabet $\bB$ and $\smx\in\{\sml,\smr\}$, the
	words $w$ and $\std_\smx(w)$ have equivalent $\smx$\ps\ insertions.
\end{lem}

There are also corresponding \emph{standardization and
de-standardization} processes for \ps\ tableaux. Considering a \ps\ tableau
$R$, $\Std_{\sml}(R)$ denotes the \lps\ tableau obtained from $R$, by reading
its entries column by column, from left to right, and on each column from top to
bottom, and attaching to each symbol $a\in\cont(R)$ an index $i\in\aA$ to the 
$i$-th
appearance of $a$. 
Moreover, $\Std_{\smr}(R)$ denotes the \rps\ tableau obtained from $R$,
by reading its entries column by column, from right to left, and on each column
from bottom to top, attaching to each symbol $a\in\cont(R)$ an index $i\in \aA$ 
to the
$i$-th appearance of $a$. 

On the other hand, if the considered \ps\ tableau $R$ has symbols from
$\mathcal{C}(\bB)$, for some alphabet $\bB$, then the \emph{de-standardization of
$R$}, denoted by $\Dstd(R)$,
is the tableau produced by erasing the indexes of each of its underlying
symbols.

A direct consequence of the lemma is the following:
\begin{remark}\label{rmk:std_insertion}
If $R=\pP_\xps(w)$, for some word $w$, 
then $\pP_\xps(\std_\smx(w))=\Std_\smx(R)$.
\end{remark}

 We prove the following result.

\begin{lem}\label{rmk:std_surjection}
	For $\smx\in\{\sml,\smr\}$, let $R$ be a $\smx$PS tableau over an 
alphabet $\bB$ and let $\Std_\smx(R)$ be the corresponding standardized 
$\smx$PS tableau over the alphabet $\mathcal{C}(\bB)$. Any
	word that inserts to $\Std_\smx(R)$ using 
Algorithm~\ref{alg:PSinsertion} has the form $\std_\smx(w)$, for
	some word $w$ over $\bB$. 
\end{lem}
\begin{proof}
	Since the \rps\ case follows with a similar argument, we only prove the 
\lps\ case.
It is clear that if $z$ is a word over $\mathcal{C}(\bB)$ which inserts to 
$\Std_{\sml}(R)$ using 
Algorithm~\ref{alg:PSinsertion}, then there exits a word $w$ over  $\bB$, where 
$w$ is obtained from $z$ removing the indexes of all symbols and 
$\cont(z)=\cont(\std_{\sml}(w))$.

Suppose 
that there are symbols
	$a_i,a_j$ with $i>j$ such that $a_i$ occurs to the left of $a_j$ in 
$z$. Then,
	using Algorithm~\ref{alg:PSinsertion}, the symbol $a_j$ is going to 
appear in $\Std_{\sml}(R)$
	either in the same column of $a_i$ below it, or in a column to the 
right. This contradicts the process of \lps\ standardization of a tableaux
	previously described, which follows on the columns of $R$ from left to 
right and top to bottom. It follows that  $a_j$ occurs in a column to the left 
where $a_i$ occurs.
	Therefore, the only words that can insert to $\Std_\smx(R)$ have the 
required form.
\end{proof}

The following result is a restatement of \cite[Proposition~3.5]{1706.06884}.

\begin{prop}\label{prop:std_and_Dstd}
	For any word $w$ over an alphabet $\bB$ and $\smx\in\{\sml,\smr\}$, we have
	\begin{enumerate}
		\item $\pP_\xps(\std_\smx(w))=\Std_\xps(\pP_\xps(w))$; and
		\item $\Dstd\big(\pP_\xps\big(\std_\smx(w)\big)\big)=\pP_\xps(w)$.
	\end{enumerate}
\end{prop}


\section{A Robinson--Schensted--Knuth-like correspondence}
\label{sec:rsklike}

In its original formulation, the Robinson correspondence is a bijection
that maps standard words from $\mathfrak{S}(\aA_n)$ to pairs of standard Young
tableaux of the same shape with content $\aA_n$. For each
standard word $\sigma$ of $\mathfrak{S}(\aA_n)$, the first component of the
pair of tableaux is obtained by applying Schensted's insertion
algorithm to $\sigma$, whereas the second is the tableaux that
records the position of each symbol inserted into the first tableaux.
The Robinson--Schensted correspondence is the generalization of this
correspondence to words. Both can be seen as particular cases of the
Robinson--Schensted--Knuth correspondence, which is the bijection obtained by
considering two-rowed arrays in lexicographic order as starting point and
taking as images pairs of semistandard Young tableaux, such that the first
tableau is obtained from the Schensted insertion applied to the second
component of the two-rowed array and the second tableau is obtained from
recording the steps of this insertion according to the first word of the 
two-rowed array
(see \cite{MR1464693} for more details concerning these constructions).

In this section we introduce Robinson--Schensted--Knuth-like correspondences for
both \lps\ and \rps\ tableaux. These correspondences will map two-rowed arrays
over a totally ordered alphabet to pairs of \ps\ tableaux having the same
shape and avoiding certain pairs of patterns. Similarly to what happens with the
original correspondence, the first element of the pair will be a \ps\
tableau coming from the \ps\ insertion of the second word of the
two-rowed array, and the second element, the tableau which records 
this insertion according to the first word of the two-rowed array. To prove this 
result
we will make use of \cite[Theorem~3.9]{BL2005} in which the authors establish a
Robinson-like correspondence between standard words from $\mathfrak{S}(\aA_n)$ and
pairs of standard Patience Sorting tableaux over $\aA_n$ having the same
shape and avoiding certain pairs of patterns.

%
%


\subsection{A Robinson--Schensted-like correspondence}

We start by presenting a Robinson--Schensted-like correspondence. In order to
do so, first we provide an algorithm which is the adaptation of
\cite[Algorithm~3.1]{BL2005} to words over a totally ordered alphabet.

\begin{algorithm}[Extended \ps\ insertion of a word]\label{alg:ExtendedPS}
	~\par\nobreak
	\textit{Input:} A word $w$ over an alphabet.

	\textit{Output:} A pair of tableaux $(\pPl(w),\qQl(w))$ (resp.\ $(\pPr(w),
	\qQr(w))$) of the same shape.

	\textit{Method:}
	\begin{enumerate}
		\item If $w=\varepsilon$, output a pair of empty tableaux $(\emptyset,
		\emptyset)$. Otherwise:
		\item $w=w_1\cdots w_n$, with symbols $w_1,\ldots,w_n$ of the 
alphabet.
		Let
		\begin{align*}
		\big(\pPl(w_1),\qQl(w_1)\big)=\left(\ytableausetup
		{boxsize=1.1em, aligntableaux=center}\begin{ytableau}
		w_1 \end{ytableau},\begin{ytableau}
		1 \end{ytableau}\right) =\big(\pPr(w_1),\qQr(w_1)\big),
		\end{align*}
		and for each remaining symbol $w_j$ with $1<j\leq n$,
		denote by $r_1\leq \dots\leq r_k$ (resp.\ $r_1< \dots< r_k$) 
the
		symbols in the bottom row of the tableau $\pPl(w_1\cdots w_{j-1})$
		(resp.\ $\pPr(w_1\cdots w_{j-1})$). Then:
		\begin{itemize}
                \item if $r_k\leq w_j$ (resp. $r_k < w_j$), simultaneously 
attach new boxes, one to  $\pPl(w_1\cdots 
w_{j-1})$
                (resp.\ $\pPr( w_1\cdots w_{j-1})$) and another to 
$\qQl(w_1\cdots w_{j-1})$ (resp.\ $\qQr(w_1\cdots
                w_{j-1})$), at the right of the bottom row, and fill the 
first with the symbol $w_j$ and the second  with the symbol $j$;

                \item otherwise, if $m=\min\{i\in\{1,\ldots, k\}: w_j<
r_i\}$,
                (resp.\ $m=\min\{i\in\{1,\ldots, k\}: w_j\leq r_i\}$)
                simultaneously attach  one new  box on the top of each of the
                $m$-th columns of $\pPl(w_1\cdots w_{j-1})$ (resp.\ $\pPr(
                w_1\cdots w_{j-1})$) and $\qQl(w_1 \cdots w_{j-1})$ (resp.\ 
                $\qQr(w_1\cdots w_{j-1})$). Then, insert $j$ in the new box of
                $\qQl(w_1\cdots w_{j-1})$ (resp.\ $\qQr(w_1 \cdots w_{j-1})$)
                and in $\pPl(w_1\cdots w_{j-1})$ (resp.\ $\pPr(w_1\cdots 
                w_{j-1})$) bump all the symbols of the column containing $r_m$
                to the box above and insert $w_j$ in the box which has been
                cleared and previously contained the symbol $r_m$.
		\end{itemize}
		Output the resulting pair of tableaux.
	\end{enumerate}
\end{algorithm}


For brevity, whenever no distinction is needed, $(\pP(w),\qQ(w))$ will denote
one of the pairs of tableaux $(\pPl(w),\qQl(w))$ or $(\pPr(w),\qQr(w))$, 
obtained from the insertion of a word $w$ under
the previous algorithm. Note that the symbols of the \ps\ tableau $\pP(w)$
obtained from the insertion of $w$ under the previous algorithm are precisely
the symbols occurring in $w$, which is taken over an arbitrary alphabet,
whereas the symbols occurring in $\qQ(w)$ are precisely the symbols from
$\aA_{\abs{w}}$. Moreover, $P(w)$ coincides with the \ps\ tableau obtained
from the insertion of $w$ under Algorithm~\ref{alg:PSinsertion} and $\qQ(w)$
is always a recording tableau. The difference between this algorithm and
Algorithm~3.2 of \cite{1706.06884} is that we are simultaneously inserting
symbols in these tableaux and recording the construction of each new box in
$\pP(w)$ using the recording tableau $\qQ(w)$.

Guided by the example of the plactic monoid, our goal is now to show that starting with a pair of tableaux $(P,Q)$ of
the same shape, where $P$ is a \ps\ tableau and $Q$ a recording tableau with content $\aA_n$ (where $n$ is the number of
entries in $P$ or $Q$), we are able to associate a unique word whose insertion under Algorithm~\ref{alg:ExtendedPS} yields
$(P,Q)$.

First, we note that for any pair of tableaux of the form $\big(\pP(w),
\qQ(w)\big)$, we will be able to trace back the word $w$, whose
insertion under Algorithm~\ref{alg:ExtendedPS} led to this pair of tableaux.

The insertion of a given word $w=w_1\cdots w_n$ under
Algorithm~\ref{alg:PSinsertion} is done by inserting each of its symbols, from
left to right in the previously obtained tableaux (starting with the empty
tableaux $\emptyset$). Thus, using the notation
from \cite[after Algorithm~3.2]{1706.06884}, it makes sense to denote the
insertion of $w$ in the following way:
\begin{equation*}
\pP(w)=\pP(w_1w_2\cdots w_n)=\big((\emptyset\leftarrow w_1)\leftarrow w_2\big)
\leftarrow\ldots\leftarrow w_n.
\end{equation*}
Since in the case of Algorithm~\ref{alg:ExtendedPS}, we are simultaneously
making two insertions, it makes sense to extend this
notation to the following:
\begin{equation*}
\big(\pP(w),\qQ(w)\big)=\Big(\big((\emptyset,\emptyset)\leftarrow (w_1,1)\big)
\leftarrow(w_2,2)\Big)\leftarrow\ldots\leftarrow (w_n,n).
\end{equation*}
With this notation, the steps of the extended \lps\ insertion of the word $w=4 6 
2 3 2 1 4$ are
shown in Figure \ref{figure:extended_insertion}.
\begin{figure}[t]
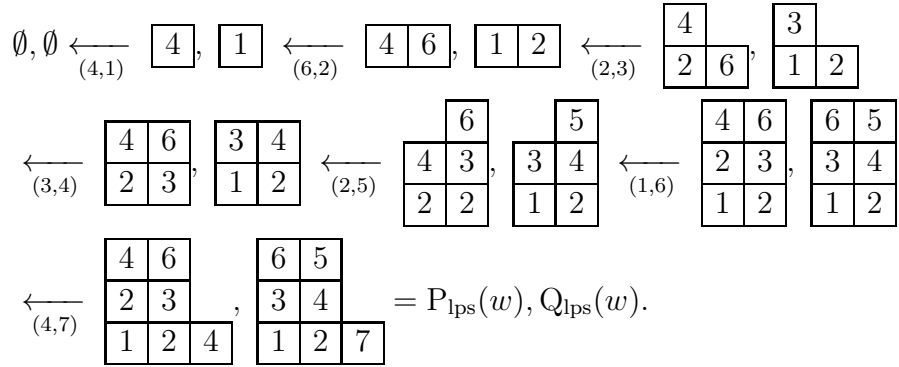

	\centering
	\begin{align*}
	&\emptyset,\emptyset \xleftarrow[(4,1)]{\ \ \ }\
	\ytableausetup
	{mathmode, boxsize=1.3em, aligntableaux=center}
	\begin{ytableau}
	4
	\end{ytableau},\
	\begin{ytableau}
	1
	\end{ytableau}\
	\xleftarrow[(6,2)]{\ \ \ }\
	\begin{ytableau}
	4 & 6
	\end{ytableau},\
	\begin{ytableau}
	1 & 2
	\end{ytableau}\
	\xleftarrow[(2,3)]{\ \ \ }\
	\begin{ytableau}
	4 & \none\\
	2 & 6
	\end{ytableau},\
	\begin{ytableau}
	3 & \none\\
	1 & 2
	\end{ytableau}\\
	&\xleftarrow[(3,4)]{\ \ \ }\
	\begin{ytableau}
	4 & 6\\
	2 & 3
	\end{ytableau},\
	\begin{ytableau}
	3 & 4\\
	1 & 2
	\end{ytableau}\
	\xleftarrow[(2,5)]{\ \ \ }\
	\begin{ytableau}
	\none & 6 \\
	4 & 3 \\
	2 & 2
	\end{ytableau},\
	\begin{ytableau}
	\none & 5 \\
	3 & 4 \\
	1 & 2
	\end{ytableau}\
	\xleftarrow[(1,6)]{\ \ \ }\
	\begin{ytableau}
	4 & 6 \\
	2 & 3 \\
	1 & 2
	\end{ytableau},\
	\begin{ytableau}
	6 & 5 \\
	3 & 4 \\
	1 & 2
	\end{ytableau}\\
	&\xleftarrow[(4,7)]{\ \ \ }\
	\begin{ytableau}
	4 & 6 & \none\\
	2 & 3 & \none\\
	1 & 2 & 4
	\end{ytableau},\
	\begin{ytableau}
	6 & 5 & \none\\
	3 & 4 & \none\\
	1 & 2 & 7
	\end{ytableau}
	=\pPl(w),\qQl(w).\
	\end{align*}
	\caption{Extended \lps\ insertion of the word $w=4 6 2 3 2 1 4$,
		where the first coordinate of the pair below the arrow indicates
		the symbol being inserted and the second coordinate the step at which
		this symbol is being inserted.}
	\label{figure:extended_insertion}
\end{figure}

Note that each time we insert a symbol into the bottom row of $\pP(w)$ using
extended \ps\ insertion, a symbol recording this insertion is simultaneously
inserted on top of the corresponding column in $\qQ(w)$.
Denote by $\qQ'(w)$ the tableau obtained by reversing the
columns of $\qQ(w)$. Reading the entries of $\pP(w)$ according to the order 
determined by
$\qQ'(w)$ allows us to get the word we started with.
From Figure \ref{figure:extended_insertion}, we know that if
$w=4 6 2 3 2 1 4$, then
\begin{equation*}
\ytableausetup
{mathmode, boxsize=1.3em, aligntableaux=center}
\pPl(w)=\begin{ytableau}
4 & 6 & \none\\
2 & 3 & \none\\
1 & 2 & 4
\end{ytableau}\ \ \text{ and }\ \
\qQ'_\bell(w)=\begin{ytableau}
1 & 2 & \none\\
3 & 4 & \none\\
6 & 5 & 7
\end{ytableau}.
\end{equation*}
Figure \ref{figure:reading_tableaux} describes the process of reading
$\pP_\bell(w)$ according to $\qQ_\bell'(w)$.

\begin{figure}[t]
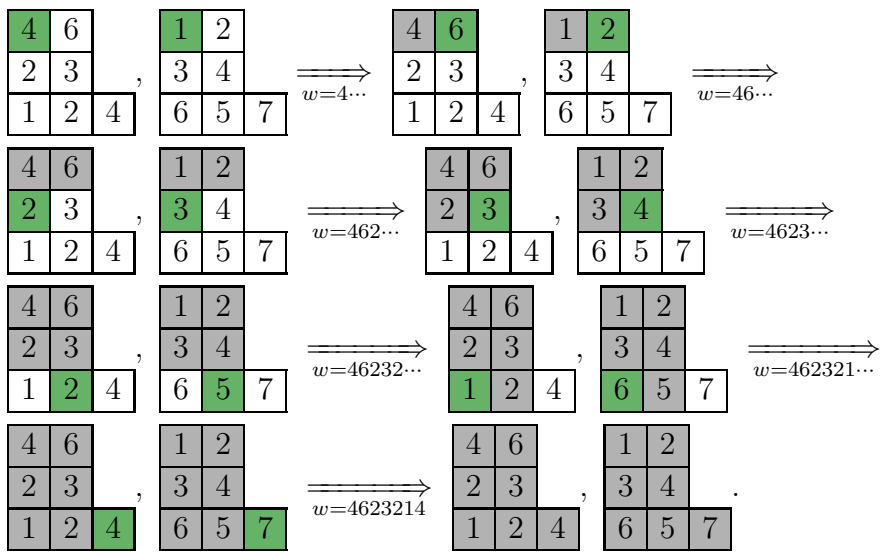

	\centering
	\begin{align*}
	&\begin{ytableau}
	*(Green!60!white)4 & 6 & \none\\
	2 & 3 & \none\\
	1 & 2 & 4
	\end{ytableau},\
	\begin{ytableau}
	*(Green!60!white)1 & 2 & \none\\
	3 & 4 & \none\\
	6 & 5 & 7
	\end{ytableau} \xRightarrow[w=4\cdots]{\ \ \ }\
	\ytableausetup
	{mathmode, boxsize=1.3em, aligntableaux=center}
	\begin{ytableau}
	*(Grey!60!white)4 & *(Green!60!white)6 & \none\\
	2 & 3 & \none\\
	1 & 2 & 4
	\end{ytableau},\
	\begin{ytableau}
	*(Grey!60!white)1 & *(Green!60!white)2 & \none\\
	3 & 4 & \none\\
	6 & 5 & 7
	\end{ytableau}\
	\xRightarrow[w=4 6\cdots]{\ \ \ }\\
	&\begin{ytableau}
	*(Grey!60!white)4 & *(Grey!60!white)6 & \none\\
	*(Green!60!white)2 & 3 & \none\\
	1 & 2 & 4
	\end{ytableau},\
	\begin{ytableau}
	*(Grey!60!white)1 & *(Grey!60!white)2 & \none\\
	*(Green!60!white)3 & 4 & \none\\
	6 & 5 & 7
	\end{ytableau}\
	\xRightarrow[w=4 6 2\cdots]{\ \ \ }\
	\begin{ytableau}
	*(Grey!60!white)4 & *(Grey!60!white)6 & \none\\
	*(Grey!60!white)2 & *(Green!60!white)3 & \none\\
	1 & 2 & 4
	\end{ytableau},\
	\begin{ytableau}
	*(Grey!60!white)1 & *(Grey!60!white)2 & \none\\
	*(Grey!60!white)3 & *(Green!60!white)4 & \none\\
	6 & 5 & 7
	\end{ytableau}\
	\xRightarrow[w=4 6 2 3\cdots]{\ \ \ }\\
	&\begin{ytableau}
	*(Grey!60!white)4 & *(Grey!60!white)6 & \none\\
	*(Grey!60!white)2 & *(Grey!60!white)3 & \none\\
	1 & *(Green!60!white)2 & 4
	\end{ytableau},\
	\begin{ytableau}
	*(Grey!60!white)1 & *(Grey!60!white)2 & \none\\
	*(Grey!60!white)3 & *(Grey!60!white)4 & \none\\
	6 & *(Green!60!white)5 & 7
	\end{ytableau}\
	\xRightarrow[w=4 6 2 3 2\cdots]{\ \ \ }\
	\begin{ytableau}
	*(Grey!60!white)4 & *(Grey!60!white)6 & \none\\
	*(Grey!60!white)2 & *(Grey!60!white)3 & \none\\
	*(Green!60!white)1 & *(Grey!60!white)2 & 4
	\end{ytableau},\
	\begin{ytableau}
	*(Grey!60!white)1 & *(Grey!60!white)2 & \none\\
	*(Grey!60!white)3 & *(Grey!60!white)4 & \none\\
	*(Green!60!white)6 & *(Grey!60!white)5 & 7
	\end{ytableau}\
	\xRightarrow[w=4 6 2 3 2 1\cdots]{\ \ \ }\\
	&\begin{ytableau}
	*(Grey!60!white)4 & *(Grey!60!white)6 \\
	*(Grey!60!white)2 & *(Grey!60!white)3 \\
	*(Grey!60!white)1 & *(Grey!60!white)2 & *(Green!60!white)4
	\end{ytableau},\
	\begin{ytableau}
	*(Grey!60!white)1 & *(Grey!60!white)2 \\
	*(Grey!60!white)3 & *(Grey!60!white)4 \\
	*(Grey!60!white)6 & *(Grey!60!white)5 & *(Green!60!white)7
	\end{ytableau}\
	\xRightarrow[w=4 6 2 3 2 1 4]{\ \ \ }\
	\begin{ytableau}
	*(Grey!60!white)4 & *(Grey!60!white)6 & \none\\
	*(Grey!60!white)2 & *(Grey!60!white)3 & \none\\
	*(Grey!60!white)1 & *(Grey!60!white)2 & *(Grey!60!white)4
	\end{ytableau},\
	\begin{ytableau}
	*(Grey!60!white)1 & *(Grey!60!white)2 & \none\\
	*(Grey!60!white)3 & *(Grey!60!white)4 & \none\\
	*(Grey!60!white)6 & *(Grey!60!white)5 & *(Grey!60!white)7
	\end{ytableau}
	.\
	\end{align*}
	\caption{Process of obtaining $w=4623214$ from reading $\pPl(4623214)$
		according to the order given by $\qQ'_\bell(4623214)$.}
	\label{figure:reading_tableaux}
\end{figure}



The \emph{column reading}, denoted $\cC(R )$, of a tableau $R$ is the word 
obtained by
proceeding through the columns, from leftmost to rightmost, and reading each 
column from top to bottom. For example, the column reading of the tableau 
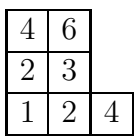
\begin{equation*}
\ytableausetup
{mathmode, boxsize=1.3em, aligntableaux=center}
\begin{ytableau}
4 & 6 & \none\\
2 & 3 & \none\\
1 & 2 & 4
\end{ytableau}
\end{equation*}
is the word $v=421\, 632\,4$. Note that
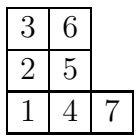
\begin{equation*}
\ytableausetup
{mathmode, boxsize=1.3em, aligntableaux=center}
\begin{ytableau}
3 & 6 & \none\\
2 & 5 & \none\\
1 & 4 & 7
\end{ytableau}.
\end{equation*}
is the recording tableau $\qQ_\bell(v)$ of the word $v$.

However, it is not guaranteed that we can start from an arbitrary pair of \ps\ tableaux $(R,S)$ with the same
shape, where $R$ is semistandard and $S$ is a recording tableaux,
read $R$ according to $S$, and obtain a word that inserts to
$(R,S)$ under Algorithm~\ref{alg:ExtendedPS}. For instance, considering the pair
of \lps\ tableaux
\begin{equation*}
\ytableausetup
{mathmode, boxsize=1.1em, aligntableaux=center}
\big(R,S\big)=\left(\begin{ytableau}
3 & \none\\
2 & \none\\
1 & 1
\end{ytableau},\
\begin{ytableau}
4 & \none\\
3 & \none\\
1 & 2
\end{ytableau}\right)
\end{equation*}
then reading $R$ according to $S'$ leads to the word $w=3121$.
Using extended \lps\ insertion, this word inserts to the pair
\begin{equation*}
\Big(\pPl(w),\qQl(w)\Big)=\left(\begin{ytableau}
3 & 2\\
1 & 1
\end{ytableau},\
\begin{ytableau}
2 & 4\\
1 & 3
\end{ytableau}\right)\neq\left(\begin{ytableau}
3 & \none\\
2 & \none\\
1 & 1
\end{ytableau},\
\begin{ytableau}
4 & \none\\
3 & \none\\
1 & 2
\end{ytableau}\right).
\end{equation*}
With another simple example we can conclude that the same problem can occur
when considering  pairs of \rps\ tableaux.

Within the setting of standard words
and pairs of standard \ps\ tableaux, this problem led
the authors of \cite{BL2005} to restrict the pairs of standard \ps\
tableaux that can be considered by introducing  the concept
of \emph{stable pairs set}.

In order to introduce this notion, first
we need the concept of pattern avoidance, which is also defined in \cite{BL2005}.
So, given standard words $u$ and $v$ with $\abs{v}=m
\leq n=\abs{u}$, we say that $u$ \emph{contains the pattern} $v$, if there exists
a subword $u_{i_1}\cdots u_{i_m}$ of $u$ of length $m$ that is order isomorphic
to $v$. Otherwise, we say that $u$ \emph{avoids the pattern} $v$.

Note that in the above  definition, the symbols $u_{i_1},\ldots, u_{i_m}$
are not required to be contiguous. However, the definition that we will adopt
henceforth is that, unless a dash is inserted in $v$ indicating which of the symbols $u_{i_1}, \ldots, u_{i_m}$ are not required to be contiguous, they have to be contiguous. 


For example 
\begin{itemize}
	\item the standard word $\sigma=3142$ contains exactly one occurrence of a
	$2$-$31$ pattern given by the subword $u= 342$;
	\item the standard word $\stdl(2312)=2_1 3_1 1_1 2_2$ contains exactly one
	occurrence of a $2$-$13$ pattern given by the subword $u = 2_1 1_1
	2_2$.
\end{itemize}

The \emph{standard stable pairs set over an alphabet}, is the set composed by
the pairs of \ps\ tableaux $(R,S)$ of the same shape,
such that $R$ is a standard tableaux over that alphabet, $S$ is a recording
tableau and the pair of column readings $(\cC(R), \cC(S'))$ avoids simultaneous
occurrences of the pairs of patterns $(31\text{-}2,13\text{-}2)$, $(31\text{-}2,
23\text{-}1)$ and $(32 \text{-}1,13\text{-}2)$ at the same positions of $\cC(R)$
and $\cC(S')$.



Within this setting, using Algorithm~\ref{alg:ExtendedPS}, Theorem~3.9 of
\cite{BL2005} can be restated in the following way

\begin{prop}\label{prop:std_rob_schensted}
	There is a one-to-one correspondence between the set of
	standard words over an alphabet and the standard stable pairs set over
	that alphabet.
	The one-to-one correspondence is given under the mapping
	\[ w\mapsto (\pP(w),\qQ (w)).\]
\end{prop}

 Since we aim to present a Robinson--Schensted-like correspondence, it is
natural to consider the following generalization of the standard stable set
previously introduced. For $\smx\in\{\sml,\smr\}$, the $\smx$\ps\ \emph{stable pairs
set over an alphabet $\bB$} is the set composed by the pairs of $\smx$\ps\ tableaux
$(R,S)$ such that $\big(\Std_{\smx}\big(R\big),S\big)$ is a standard stable 
pair over $\mathcal{C}(\bB)$.


We are now able to introduce the Robinson--Schensted-like correspondence for the
words case.

\begin{thm}\label{thm:RSK}
	For $\smx\in \{\sml,\smr\}$, there is a one-to-one correspondence 
between the	 set
	of words over an alphabet and pairs of $\smx$\ps\ tableaux of the $\smx$\ps\
	stable pairs set over that alphabet. The one-to-one correspondence is given
	under the mapping
	\[ w\mapsto (\pP_\xps(w),\qQ_\xps(w)).\]
\end{thm}

\begin{proof}
	Let $\bB$ be an alphabet and $\smx\in\{\sml,\smr\}$. Given a word $w$ 
over $\bB$, 
	we obtain a unique standard word $\std_\smx(w)$ over 
$\mathcal{C}(\bB)$. Using 
	Proposition~\ref{prop:std_rob_schensted}, we obtain a pair $(\pP_\xps(
	\std_\smx(w)), \qQ_\xps(\std_\smx(w)))$ in the standard stable pairs set over $\mathcal{C}(\bB)$.
	By Proposition~\ref{prop:std_and_Dstd}~(1), $\pP_\xps(\std_\smx(w))=
	\Std_{\smx}(\pP_\xps(w))$. Since $w$ and $\std_\smx(w)$ have equivalent 
$\smx$\ps\
	insertions, $\qQ_\xps(\std_\smx(w))=\qQ_\xps(w)$. Thus, it follows that $(\pP_\xps(\std_\smx(w)),\qQ_\xps(\std_\smx(w)
	))=(\Std_{\smx}(\pP_\xps(w)),\qQ_\xps(w))$. So, we obtain a pair of 
$\smx$\ps\ tableaux
	$(\pP_\xps(w),\qQ_\xps(w))$ in the $\smx$\ps\ stable pairs set over $\bB$. Note that
	different words lead to different pairs.
	
	Given a pair of $\smx$\ps\ tableaux $(R,S)$
	in the $\smx$\ps\ stable pairs set over $\bB$, by definition 
$(\Std_{\smx}(R),S)$
	is in the standard stable pairs set over $\mathcal{C}(\bB)$. By
	Proposition~\ref{prop:std_rob_schensted}, there exists a standard word over 
	$\mathcal{C}(\bB)$ that
	inserts to $(\Std_{\smx}(R),S)$. From Lemma~\ref{rmk:std_surjection}, 
this word
	must have the form $\std_\smx(w)$ for some word $w$ over $\bB$. 
Thus,
	$(\pP_\xps(\std_\smx(w)),\qQ_\xps(\std_\smx(w)))=(\Std_{\smx}(R),S)$. 
Since by Proposition~\ref{prop:std_and_Dstd}~(1), $\pP_\xps(\std_\smx(w))=
	\Std_{\smx}(\pP_\xps(w))$, and because $\qQ_\xps(\std_\smx(w))=
	\qQ_\xps(w)$, it follows that
	$(\Std_{\smx}(\pP_\xps(w)),\qQ_\xps(w))=(\Std_{\smx}(R),S)$. Applying 
$\Dstd()$ to both
	$\Std_{\smx}(\pP_\xps(w))$ and $\Std_{\smx}(R)$, we have 
$(\pP_\xps(w),\qQ_\xps(w))=
	(R,S)$. Therefore under Algorithm~\ref{alg:ExtendedPS}, the word $w=\dstd(\std_\smx(w))$ 
	over $\bB$ inserts to the pair $(R,S)$.
\end{proof}

\subsection{A Robinson--Schensted--Knuth-like correspondence}

It is common to present a standard word $\sigma=\sigma_1\sigma_2\cdots\sigma_k$
in two-line notation as
\begin{equation*}
\sigma=\begin{pmatrix}
1 & 2 & \cdots & k\\
\sigma_1 & \sigma_2 & \cdots & \sigma_k
\end{pmatrix}.
\end{equation*}
This notation can be extended to arbitrary words, $w=w_1w_2\cdots w_k$, with 
symbols $w_1,w_2,\ldots,w_k$, in the following way:
\begin{equation*}
w=\begin{pmatrix}
1 & 2 & \cdots & k\\
w_1 & w_2 & \cdots & w_k
\end{pmatrix}.
\end{equation*}
Looking to Algorithm~\ref{alg:ExtendedPS} from this perspective, we observe
that it just describes the simultaneous insertion of the words $w=w_1w_2\cdots w_k$  and
$12\cdots k$, where the standard word $12\cdots k$ is inserted according
to the \ps\ insertion of $w$.

Hence, seeing things from the two row notation perspective, 
Proposition~\ref{prop:std_rob_schensted}
establishes a bijection between standard words $\sigma=\begin{psmallmatrix}
1 & 2 & \cdots &  k\\
\sigma_1 & \sigma_2 & \cdots & \sigma_k
\end{psmallmatrix}$ and pairs of standard \ps\ tableaux $\big(\pP(\sigma),
\qQ(\sigma)\big)$ in the standard stable pairs set.
In the same way, Theorem~\ref{thm:RSK} establishes a bijection between words
$w= \begin{psmallmatrix}
1 & 2 & \cdots &  k\\
w_1 & w_2 & \cdots & w_k
\end{psmallmatrix}$ and pairs of standard \ps\ tableaux $\big(\pP(w),
\qQ(w)\big)$ in the stable pairs set. This perspective suggests a
generalization of Algorithm~\ref{alg:ExtendedPS} to two-rowed arrays 
$\begin{psmallmatrix}
u\\ v \end{psmallmatrix}$, where $u$ and $v$ are words of the same length and a consequent
generalization of both the Proposition~\ref{prop:std_rob_schensted} and the
Theorem~\ref{thm:RSK}.


So, henceforth a \emph{two-rowed array over an alphabet} will be an array 
$\begin{psmallmatrix} u\\
v \end{psmallmatrix}$ where $u$ and $v$ are arbitrary words of the same length
over that alphabet (note that word arrays are not biwords in the
sense of \cite{MR1905123}). We propose the following generalization of
Algorithm~\ref{alg:ExtendedPS}:

\begin{algorithm}[Extended \ps\ insertion of a word array]\label{alg:ArrayExtendedPS}
	~\par\nobreak
	\textit{Input:} A two-rowed array  $w=\begin{psmallmatrix} u\\ v 
\end{psmallmatrix}$ over an alphabet.

	\textit{Output:} A pair of tableaux $(\pPl(w),\qQl(w))$ (resp.\ $(\pPr(w),\qQr(w))$) of the same shape.

	\textit{Method:}
	\begin{enumerate}
		\item If $v=\varepsilon$ and $u=\varepsilon$, output a pair of empty
		tableaux $(\emptyset,\emptyset)$. Otherwise:
		\item $v=v_1\cdots v_n$ and $u=u_1\cdots u_n$, for symbols $v_1,\ldots,v_n$  and $u_1,\ldots,u_n$. Set
		\[
		\Big(\pPl\big(\begin{psmallmatrix} u_1\\ v_1 \end{psmallmatrix}\big),\qQl\big(\begin{psmallmatrix} u_1\\ v_1 \end{psmallmatrix}\big)\Big)=\left(\ytableausetup
		{boxsize=1.1em, aligntableaux=center}\begin{ytableau}
		v_1 \end{ytableau},\begin{ytableau}
		u_1 \end{ytableau}\right) =\Big(\pPr\big(\begin{psmallmatrix} u_1\\ v_1 \end{psmallmatrix}\big),\qQr\big(\begin{psmallmatrix} u_1\\ v_1 \end{psmallmatrix}\big)\Big),
		\]
		and for each remaining symbol $v_j$ with $1<j\leq n$,
		denote by $r_1\leq \dots\leq r_k$ (resp.\ $r_1< \dots< r_k$) 
the
		symbols in the bottom row of the tableau $\pPl\big(\begin{psmallmatrix} u_1 & \cdots & u_{j-1}\\ v_1 & \cdots &
		v_{j-1} \end{psmallmatrix}\big)$
		(resp.\ $\pPr\big(\begin{psmallmatrix} u_1 & \cdots & u_{j-1}\\ v_1 & \cdots &
		v_{j-1} \end{psmallmatrix}\big)$). Then:
		\begin{itemize}
			\item if $r_k \leq v_j$ (resp.\ $r_k < v_j$), 
simultaneously attach two new boxes, one to $\pPl\big(\begin{psmallmatrix} u_1 & 
\cdots & u_{j-1}\\ v_1 & \cdots &
			v_{j-1} \end{psmallmatrix}\big)$ (resp.\ $\pPr\big(\begin{psmallmatrix} u_1 & \cdots & u_{j-1}\\ v_1 & \cdots &
			v_{j-1} \end{psmallmatrix}\big)$) and the other to 
$\qQl\big(\begin{psmallmatrix} u_1 & \cdots & u_{j-1}\\ v_1 & \cdots &
			v_{j-1} \end{psmallmatrix}\big)$ (resp.\ $\qQr\big(\begin{psmallmatrix} u_1 & \cdots & u_{j-1}\\ v_1 & \cdots &
			v_{j-1} \end{psmallmatrix}\big)$), at the right of the 
bottom row, and fill the first with the symbol $v_j$ and the second with the 
symbol $u_j$;

			\item otherwise, if $m=\min\{i\in\{1,\ldots, k\}: v_j <
r_i\}$,
			(resp.\ $m=\min\{i\in\{1,\ldots, k\}: v_j\leq 
r_i\}$)
			simultaneously attach one new box on top of each of 
the $m$-th columns of
			$\pPl\big(\begin{psmallmatrix} u_1 & \cdots & u_{j-1}\\ v_1 & \cdots &
			v_{j-1} \end{psmallmatrix}\big)$ (resp.\ $\pPr\big(\begin{psmallmatrix} u_1 & \cdots & u_{j-1}\\ v_1 & \cdots &
			v_{j-1} \end{psmallmatrix}\big)$) and $\qQl\big(\begin{psmallmatrix} u_1 & \cdots & u_{j-1}\\ v_1 & \cdots &
			v_{j-1} \end{psmallmatrix}\big)$ (resp.\ $\qQr\big(\begin{psmallmatrix} u_1 & \cdots & u_{j-1}\\ v_1 & \cdots &
			v_{j-1} \end{psmallmatrix}\big)$). Then, insert $u_j$ in the new box of $\qQl\big(\begin{psmallmatrix} u_1 & \cdots & u_{j-1}\\ v_1 & \cdots &
			v_{j-1} \end{psmallmatrix}\big)$ (resp.\ $\qQr\big(\begin{psmallmatrix} u_1 & \cdots & u_{j-1}\\ v_1 & \cdots &
			v_{j-1} \end{psmallmatrix}\big)$)
			and in $\pPl\big(\begin{psmallmatrix} u_1 & \cdots & u_{j-1}\\ v_1 & \cdots &
			v_{j-1} \end{psmallmatrix}\big)$ (resp.\ $\pPr\big(\begin{psmallmatrix} u_1 & \cdots & u_{j-1}\\ v_1 & \cdots &
			v_{j-1} \end{psmallmatrix}\big)$) bump
			all the symbols of the column containing $r_m$ to the box above and insert
			$v_j$ in the box which has been cleared and previously contained the symbol
			$r_m$.
		\end{itemize}
		Output the resulting pair of tableaux.
	\end{enumerate}
\end{algorithm}

Considering the two-rowed array $w=\begin{psmallmatrix} 1 & 2 & 1 & 3\\ 1 & 1 & 
2 & 1
\end{psmallmatrix}$, the extended \lps\ insertion of the two-rowed array under
Algorithm~\ref{alg:ArrayExtendedPS} is given by the pair of tableaux

\[\left(\pPl(w),\qQl(w)\right)=\left(\begin{ytableau}
\none & \none & 2\\
1 & 1 & 1
\end{ytableau},\
\begin{ytableau}
\none & \none & 3\\
1 & 2 & 1
\end{ytableau}\right)
\]
The first observation that comes out from this example is that if
$w$ is an arbitrary two-rowed array, then
$\qQl(w)$ is not necessarily an \lps\ tableau. In fact, considering
the same two-rowed array we can conclude that $\qQr(w)$ is also not
necessarily an \rps\ tableau.

However, we are interested in finding those two-rowed arrays from which we can 
always
obtain a pair of \ps\ tableaux.

Since for symbols $y_1<y_2$, a two-rowed array
$\begin{psmallmatrix} u\\ v \end{psmallmatrix}\in \left\{\begin{psmallmatrix}
y_2 & y_1\\ y_1 & y_2 \end{psmallmatrix},\begin{psmallmatrix}
y_2 & y_1\\ y_2 & y_1 \end{psmallmatrix},\begin{psmallmatrix}
y_2 & y_1\\ y_1 & y_1 \end{psmallmatrix}\right\}$ gives rise to a tableau
$\qQ_\xps(w)$, for $\smx\in\left\{\sml,\smr\right\}$, which is not an 
$\smx$\ps\ tableau, we conclude that in order to obtain a pair of $\smx$\ps\ 
tableau, $u$ has to be ordered
weakly increasingly. However, there are still two-rowed arrays 
$w=\begin{psmallmatrix}
u\\ v  \end{psmallmatrix}$ where $u$ is ordered weakly increasingly, such that
$\qQ_\xps(w)$ is not an $\xps$ tableau. For instance, for symbols $y_1<y_2$ and 
two-rowed arrays $w=
\begin{psmallmatrix} y_1 & y_1\\ y_2 & y_1 \end{psmallmatrix}$ and
$w'= \begin{psmallmatrix} y_1 & y_1\\
y_1 & y_2 \end{psmallmatrix}$, we have that $\qQl(w)$ is not an \lps\ tableau and
$\qQr(w')$ is not an \rps\ tableau.

This leads us to two different types of two-rowed arrays. For any two-rowed 
array
$\begin{psmallmatrix} u\\ v \end{psmallmatrix}=
\begin{psmallmatrix} u_1 & u_2 & \cdots & u_k\\ v_1 & v_2 & \cdots &
v_k \end{psmallmatrix}$, consider the following conditions:
\begin{enumerate}
	\item $u_1\leq u_2\leq \dots \leq u_k$;
	\item for $i\in \left\{1,2,\ldots,k-1\right\}$, if $u_i=u_{i+1}$,
	then $v_i\leq v_{i+1}$;
	\item for $i\in \left\{1,2,\ldots,k-1\right\}$, if $u_i=u_{i+1}$,
	then $v_{i+1}\leq v_{i}$.
\end{enumerate}
The two-rowed array $\begin{psmallmatrix} u\\ v \end{psmallmatrix}$ is said to 
be in 
\emph{lexicographic order} if conditions $(1)$ and $(2)$ hold, and that
it  is in \emph{reverse 
lexicographic order}
if conditions $(1)$ and $(3)$ hold. For simplicity, we will refer the two-rowed 
array as an $l$-two-rowed array, in the first case, and as a $r$-two-rowed 
array in the second.

\begin{lem}\label{lem:word_array_ext_ins}
	For any $\smx\in \left\{\sml,\smr\right\}$, if $w=\begin{psmallmatrix} u\\ v \end{psmallmatrix}=
	\begin{psmallmatrix} u_1 & u_2 & \cdots & u_k\\ v_1 & v_2 & \cdots &
	v_k \end{psmallmatrix}$ is an $\smx$-two-rowed array, then
	the tableau $\qQ_\xps(w)$ obtained from the extended $\smx$\ps\ insertion of
	$w$ under
	Algorithm~\ref{alg:ArrayExtendedPS} is an $\smx$\ps\ tableau.
\end{lem}
\begin{proof}
	Both the \lps\ and \rps\ cases follow by induction on the number $n$ of distinct
	symbols from $u$.

	For the \lps\ case, the induction proceeds in the following way.
	Case $n=1$. In this case $w=\begin{psmallmatrix} u\\ v 
\end{psmallmatrix}=
	\begin{psmallmatrix} u_1 & \cdots & u_1\\ v_1 & \cdots &
	v_k \end{psmallmatrix}$. Since $v_1\leq v_2\leq \dots \leq v_k$,
	by Algorithm~\ref{alg:ArrayExtendedPS}, $\left(\pPl(w),\qQl(w)\right)=
	\left(\ytableausetup
	{mathmode, boxsize=1.1em, aligntableaux=center}
	\begin{ytableau}
	v_1 & v_2
	\end{ytableau}\cdots\begin{ytableau}
	v_k
	\end{ytableau}, \begin{ytableau}
	u_1 & u_1
	\end{ytableau}\cdots\begin{ytableau}
	u_1
	\end{ytableau}\right)$ and thus $\qQl(w)$ is an \lps\ tableau.

	Fix $n\geq 1$ and suppose by induction hypothesis that the result holds for $n$.
	Let us prove the result for $n+1$. Having $n+1$ distinct symbols, $u=x_1^{i_1}
	x_2^{i_2}\cdots x_{n+1}^{i_{n+1}}$, for some symbols $x_1<x_2< \dots 
<x_{n+1}$
	and indexes $i_1,i_2,\ldots, i_{n+1}\in\mathbb{N}$. Thus
	\[w=
	\begin{pmatrix} x_1 & \cdots & x_1 & x_2 & \cdots & x_2 & \cdots & x_{n+1} \\
	v_1 & \cdots & v_{i_1} & v_{i_1+1} & \cdots & v_{i_1+i_2} & \cdots & v_k
	\end{pmatrix}.\]
	Since $v_{i_1+\cdots+i_n+1}\leq \dots \leq v_k$, according to 
Algorithm~\ref{alg:ArrayExtendedPS}, for any $m\in\{i_1+\cdots+i_n+2,\ldots,
	k\}$, $v_m$ is inserted in the bottom row in a column to the right of the
	column where $v_{m-1}$ was inserted. Therefore each of the symbols $x_{n+1}$
	from $u$ is inserted either on top of a column not containing $x_{n+1}$	or as
	the rightmost symbol of the bottom row. Since by induction hypothesis
	$\qQl\Big(\begin{psmallmatrix} x_1^{i_1} & \cdots & x_n^{i_n}\\ v_1 & \cdots &
	v_{i_1+\cdots+i_n} \end{psmallmatrix}\Big)$ is an \lps\
	tableau, from this discussion we conclude that the columns $\qQl(w)$
	are strictly decreasing top to bottom and the bottom row is weakly increasing
	from left to right. Therefore it follows that $\qQl(w)$ is an \lps\
	tableau.

	Regarding the \rps\ case, the induction proceeds in the following
	way. Case n=1. Recall that $w=\begin{psmallmatrix} u\\ v 
\end{psmallmatrix}=
	\begin{psmallmatrix} u_1 & \cdots & u_1\\ v_1 & \cdots &
	v_k \end{psmallmatrix}$. Since $v_k\leq \dots \leq v_2\leq 
v_1$, it follows
	that
	\begin{equation*}
	\pPr(w)=\begin{ytableau}
	v_k\\
	\none[\vdots]\\
	v_2\\
	v_1
	\end{ytableau}\ \text{ and thus }\
	\qQr(w)=\begin{ytableau}
	u_1\\
	\none[\vdots]\\
	u_1\\
	u_1
	\end{ytableau}.
	\end{equation*}
	So, we conclude that $\qQr(w)$ is an \rps\ tableau.

	Fix $n\geq 1$ and suppose by induction hypothesis that the result holds for $n$.
	Since $u$ contains $n+1$ symbols $x_1<x_2<\dots<x_{n+1}$, again 
$u=x_1^{i_1}
	x_2^{i_2}\cdots x_{n+1}^{i_{n+1}}$, and
	\[w=
	\begin{pmatrix} x_1 & \cdots & x_1 & x_2 & \cdots & x_2 & \cdots & x_{n+1} \\
	v_1 & \cdots & v_{i_1} & v_{i_1+1} & \cdots & v_{i_1+i_2} & \cdots & v_k
	\end{pmatrix}.\]
	According to Algorithm~\ref{alg:ArrayExtendedPS}, the symbol $v_{i_1+\cdots+
	i_n+1}$ is inserted in the bottom row of $\pPr\Big(\begin{psmallmatrix} x_1^{i_1} & \cdots & x_n^{i_n}\\ v_1 & \cdots &
v_{i_1+\cdots+i_n} \end{psmallmatrix}\Big)$,
	either as the rightmost symbol or in a column of this tableau. Since
	$v_k\leq \dots \leq v_{i_1+\cdots+i_n+1}$, for any 
$m\in \{i_1+\cdots+i_n+2,
	\ldots, k\}$, $v_m$ is going to be inserted in the bottom row below or to the
	left of the column where $v_{m-1}$ was inserted. So, the first symbol $x_{n+1}$
	of $u$ is either inserted to the right of the rightmost box, or on top of a
	column of $\qQr\Big(\begin{psmallmatrix} x_1^{i_1} & \cdots & x_n^{i_n}\\ v_1 & \cdots &
	v_{i_1+\cdots+i_n} \end{psmallmatrix}\Big)$. The
	remaining symbols $x_{n+1}$ are always inserted on top of a column.

	Since by induction hypothesis
	$\qQr\Big(\begin{psmallmatrix} x_1^{i_1} & \cdots & x_n^{i_n}\\ v_1 & \cdots &
	v_{i_1+\cdots+i_n} \end{psmallmatrix}\Big)$ is an \rps\
	tableau, the previous discussion allows us to conclude that $\qQr(w)$ is
	an \rps\ tableau.
\end{proof}

As in the previous cases, we are interested in describing the reverse process:
that is, starting from a pair of \ps\ tableaux our goal is to obtain the
unique two-rowed array that gives rise to this pair under
Algorithm~\ref{alg:ArrayExtendedPS}. The proof of
Lemma~\ref{lem:word_array_ext_ins} will allow us to describe it, and suggests that
the \lps\ and \rps\ versions of the reverse process will be different.

The \lps\ \emph{reverse insertion method} can be described as follows: let
$(R,S)$ be a pair of \lps\ tableaux of the same shape. Then, $\abs{R}=\abs{S}
=k$, for some $k\in \mathbb{N}$. Let $(R,S)=(R_k,S_k)$, and for any
$i\in\{1,\ldots,k-1\}$, let $(R_i,S_i)$ be the pair of tableaux (of
the same shape) obtained from removing in $S_{i+1}$ the box containing the largest
symbol that is farthest to the right, $u_{i+1}$, and in $R_{i+1}$ the box
containing the symbol $v_{i+1}$ that is in the bottom of the corresponding column
of $R_{i+1}$. From this process we obtain a two-rowed array 
$\begin{psmallmatrix}
u\\ v \end{psmallmatrix}= \begin{psmallmatrix} u_1 & u_2 & \cdots & u_k\\ v_1 & v_2
& \cdots & v_k \end{psmallmatrix}$ and it is clear that $u_1\leq 
u_2\leq\dots\leq
u_k$.

Moreover, assuming that  for each $i$,  $(R_i,S_i)$ is a pair of \lps\ 
tableaux, then if $u_i=u_{i+1}$ for some $i\in \{1,\ldots,k-1\}$, then  
$u_{i+1}$ and $u_i$ are in distinct columns and $u_{i+1}$ is
farthest to the right. So, if $v_{i+1}$ is the symbol in the box removed from the
bottom row of $R_{i+1}$, then $v_{i+1}$ is farther to the right than the symbol $v_i$
in the box removed from the bottom row of $P_i$. By the proof of the \lps\ case of
Lemma~\ref{lem:word_array_ext_ins}, it follows that $v_i\leq v_{i+1}$. 
Therefore, the
two-rowed array $\begin{psmallmatrix} u\\ v \end{psmallmatrix}$ obtained via 
this method is
an $l$-two-rowed array.

Similarly, consider the \rps\ \emph{reverse insertion method} described as 
follows. Let $(R,S)$ be a
pair of \rps\ tableaux of the same shape. Then, $\abs{R}= \abs{S}=k$, for
some $k\in \mathbb{N}$. Let $(R,S)=(R_k,S_k)$, and for any  $i\in\{1,\ldots,
k-1\}$, let $(R_i,S_i)$ be the pair of tableaux obtained from removing in
$S_{i+1}$ the box containing the largest symbol that is farthest to the left on 
top
of the respective column, $u_{i+1}$, and in $R_{i+1}$ the box containing the
symbol $v_{i+1}$ that is in the bottom of the corresponding column.

Again, if $\begin{psmallmatrix} u\\ v \end{psmallmatrix}=
\begin{psmallmatrix} u_1 & u_2 & \cdots & u_k\\ v_1 & v_2 & \cdots &
v_k \end{psmallmatrix}$ is the two-rowed array that is obtained via this 
process, then it
is clear that $u_1\leq u_2\leq\dots\leq u_k$. Assuming that $(R_i,S_i)$ is a 
pair of \rps\ tableaux for any $i\in\{1,\ldots,k-1\}$, if $u_i=u_{i+1}$ for some
$i\in\{1,\ldots,k-1\}$, then $u_{i+1}$ and
$u_i$ are either in distinct columns, $u_{i+1}$ being farther to the left, or
$u_{i+1}$ and $u_{i}$ are in the same column, $u_{i+1}$ being on top of $u_i$. In the
first case, if $v_{i+1}$ is the symbol in the box removed from the bottom row of
$R_{i+1}$, then $v_{i+1}$ is farther to the left than the symbol $v_i$ removed 
from
the box in the bottom row of $R_{i}$. By the proof of the \rps\ case of
Lemma~\ref{lem:word_array_ext_ins}, it follows that $v_{i+1}<v_i$. In the later case,
if $v_{i+1}$ is the symbol in the box removed from the bottom row of $R_{i+1}$, 
then
$v_{i+1}$ is below the symbol removed from the box in the bottom row of $R_i$, 
$v_i$.
By the proof of the \rps\ case of Lemma~\ref{lem:word_array_ext_ins},
$v_{i+1}\leq v_i$. Therefore, the two-rowed array
$\begin{psmallmatrix} u\\ v \end{psmallmatrix}$ obtained through this method is an
$r$-two-rowed array.

\begin{figure}[t]
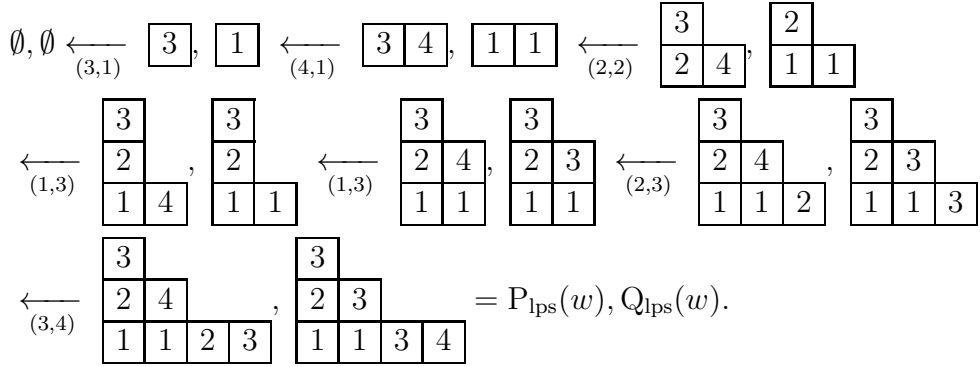

	\centering
	\begin{align*}
	&\emptyset,\emptyset \xleftarrow[(3,1)]{\ \ \ }\
	\ytableausetup
	{mathmode, boxsize=1.3em, aligntableaux=center}
	\begin{ytableau}
	3
	\end{ytableau},\
	\begin{ytableau}
	1
	\end{ytableau}\
	\xleftarrow[(4,1)]{\ \ \ }\
	\begin{ytableau}
	3 & 4
	\end{ytableau},\
	\begin{ytableau}
	1 & 1
	\end{ytableau}\
	\xleftarrow[(2,2)]{\ \ \ }\
	\begin{ytableau}
	3 & \none\\
	2 & 4
	\end{ytableau},\
	\begin{ytableau}
	2 & \none\\
	1 & 1
	\end{ytableau}\\
	&\xleftarrow[(1,3)]{\ \ \ }\
	\begin{ytableau}
	3 & \none\\
	2 & \none\\
	1 & 4
	\end{ytableau},\
	\begin{ytableau}
	3 & \none\\
	2 & \none\\
	1 & 1
	\end{ytableau}\
	\xleftarrow[(1,3)]{\ \ \ }\
	\begin{ytableau}
	3 & \none \\
	2 & 4 \\
	1 & 1
	\end{ytableau},\
	\begin{ytableau}
	3 & \none \\
	2 & 3 \\
	1 & 1
	\end{ytableau}\
	\xleftarrow[(2,3)]{\ \ \ }\
	\begin{ytableau}
	3 & \none &\none \\
	2 & 4 &\none\\
	1 & 1 & 2
	\end{ytableau},\
	\begin{ytableau}
	3 & \none \\
	2 & 3 \\
	1 & 1 & 3
	\end{ytableau}\\
	&\xleftarrow[(3,4)]{\ \ \ }\
	\begin{ytableau}
	3 & \none & \none & \none\\
	2 & 4 & \none & \none\\
	1 & 1 & 2 & 3
	\end{ytableau},\
	\begin{ytableau}
	3 & \none & \none & \none\\
	2 & 3 & \none & \none\\
	1 & 1 & 3 & 4
	\end{ytableau}
	=\pPl(w),\qQl(w).\
	\end{align*}
	\caption{Extended \lps\ insertion of the $l$-two-rowed array $w=\protect
		\begin{psmallmatrix}u\\ v \protect\end{psmallmatrix}= \protect\begin{psmallmatrix}
		1 & 1 & 2 & 3 & 3 & 3 & 4\\ 3 & 4 & 2 & 1 & 1 & 2 & 3
		\protect\end{psmallmatrix}$	under Algorithm~\ref{alg:ArrayExtendedPS}, where
		the pair below each arrow indicates the pair of symbols from $w$ 
being
		inserted.}
	\label{figure:extended_lps_insertion}
\end{figure}

Figure \ref{figure:extended_lps_insertion} shows the extended \lps\ insertion
of the $l$-two-rowed array $w=
\begin{psmallmatrix} 1 & 1 & 2 & 3 & 3 & 3 & 4\\ 3 & 4 & 2 &
1 & 1 & 2 & 3 \end{psmallmatrix}$ under Algorithm~\ref{alg:ArrayExtendedPS},
while
Figure~\ref{figure:ext_reading_tableaux} describes the
construction of the original $l$-two-rowed array $w=\protect\begin{psmallmatrix}
1 & 1 & 2 & 3 & 3 & 3 & 4\\ 3 & 4 & 2 & 1 & 1 & 2 & 3
\protect\end{psmallmatrix}$ from that pair of tableaux according to the
\lps\ reverse insertion method previously described.

\begin{figure}[t]
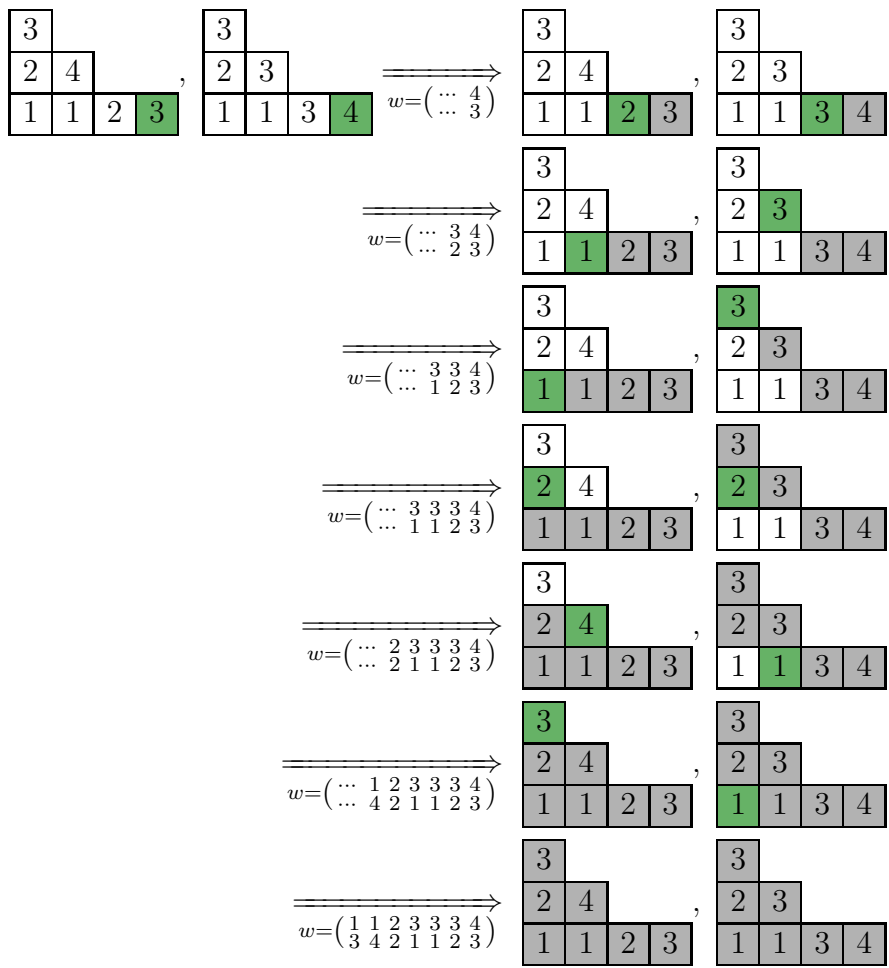

	\centering
	\begin{align*}
	\begin{ytableau}
	3 & \none & \none & \none\\
	2 & 4 & \none & \none\\
	1 & 1 & 2 & *(Green!60!white)3
	\end{ytableau},\
	\begin{ytableau}
	3 & \none & \none & \none\\
	2 & 3 & \none & \none\\
	1 & 1 & 3 & *(Green!60!white)4
	\end{ytableau} \xRightarrow[{w=\protect\begin{psmallmatrix}
		\cdots & 4 \\ \cdots & 3
		\protect\end{psmallmatrix}}]{\ \ \ }\
	\begin{ytableau}
	3 & \none & \none & \none\\
	2 & 4 & \none & \none\\
	1 & 1 & *(Green!60!white)2 & *(Grey!60!white)3
	\end{ytableau},\
	\begin{ytableau}
	3 & \none & \none & \none\\
	2 & 3 & \none & \none\\
	1 & 1 & *(Green!60!white)3 & *(Grey!60!white)4
	\end{ytableau}&&\\
	\xRightarrow[{w=\protect\begin{psmallmatrix}
		\cdots & 3 & 4\\ \cdots & 2 & 3
		\protect\end{psmallmatrix}}]{\ \ \ }\
	\begin{ytableau}
	3 & \none & \none & \none\\
	2 & 4 & \none & \none\\
	1 & *(Green!60!white)1 & *(Grey!60!white)2 & *(Grey!60!white)3
	\end{ytableau},\
	\begin{ytableau}
	3 & \none & \none & \none\\
	2 & *(Green!60!white)3 & \none & \none\\
	1 & 1 & *(Grey!60!white)3 & *(Grey!60!white)4
	\end{ytableau}&&\\
	\xRightarrow[{w=\protect\begin{psmallmatrix}
		\cdots & 3 & 3 & 4\\ \cdots & 1 & 2 & 3
		\protect\end{psmallmatrix}}]{\ \ \ }\
	\begin{ytableau}
	3 & \none & \none & \none\\
	2 & 4 & \none & \none\\
	*(Green!60!white)1 & *(Grey!60!white)1 & *(Grey!60!white)2 & *(Grey!60!white)3
	\end{ytableau},\
	\begin{ytableau}
	*(Green!60!white)3 & \none & \none & \none\\
	2 & *(Grey!60!white)3 & \none & \none\\
	1 & 1 & *(Grey!60!white)3 & *(Grey!60!white)4
	\end{ytableau}&&\\
	\xRightarrow[{w=\protect\begin{psmallmatrix}
		\cdots & 3 & 3 & 3 & 4\\ \cdots & 1 & 1 & 2 & 3
		\protect\end{psmallmatrix}}]{\ \ \ }\
	\begin{ytableau}
	3 & \none & \none & \none\\
	*(Green!60!white)2 & 4 & \none & \none\\
	*(Grey!60!white)1 & *(Grey!60!white)1 & *(Grey!60!white)2 & *(Grey!60!white)3
	\end{ytableau},\
	\begin{ytableau}
	*(Grey!60!white)3 & \none & \none & \none\\
	*(Green!60!white)2 & *(Grey!60!white)3 & \none & \none\\
	1 & 1 & *(Grey!60!white)3 & *(Grey!60!white)4
	\end{ytableau}&&\\
	\xRightarrow[{w=\protect\begin{psmallmatrix}
		\cdots & 2 & 3 & 3 & 3 & 4\\ \cdots & 2 & 1 & 1 & 2 & 3
		\protect\end{psmallmatrix}}]{\ \ \ }\
	\begin{ytableau}
	3 & \none & \none & \none\\
	*(Grey!60!white)2 & *(Green!60!white)4 & \none & \none\\
	*(Grey!60!white)1 & *(Grey!60!white)1 & *(Grey!60!white)2 & *(Grey!60!white)3
	\end{ytableau},\
	\begin{ytableau}
	*(Grey!60!white)3 & \none & \none & \none\\
	*(Grey!60!white)2 & *(Grey!60!white)3 & \none & \none\\
	1 & *(Green!60!white)1 & *(Grey!60!white)3 & *(Grey!60!white)4
	\end{ytableau}&&\\
	\xRightarrow[{w=\protect\begin{psmallmatrix}
		\cdots & 1 & 2 & 3 & 3 & 3 & 4\\ \cdots & 4 & 2 & 1 & 1 & 2 & 3
		\protect\end{psmallmatrix}}]{\ \ \ }\
	\begin{ytableau}
	*(Green!60!white)3 & \none & \none & \none\\
	*(Grey!60!white)2 & *(Grey!60!white)4 & \none & \none\\
	*(Grey!60!white)1 & *(Grey!60!white)1 & *(Grey!60!white)2 & *(Grey!60!white)3
	\end{ytableau},\
	\begin{ytableau}
	*(Grey!60!white)3 & \none & \none & \none\\
	*(Grey!60!white)2 & *(Grey!60!white)3 & \none & \none\\
	*(Green!60!white)1 & *(Grey!60!white)1 & *(Grey!60!white)3 & *(Grey!60!white)4
	\end{ytableau}&&\\
	\xRightarrow[{w=\protect\begin{psmallmatrix}
		1 & 1 & 2 & 3 & 3 & 3 & 4\\ 3 & 4 & 2 & 1 & 1 & 2 & 3
		\protect\end{psmallmatrix}}]{\ \ \ }\
	\begin{ytableau}
	*(Grey!60!white)3 & \none & \none & \none\\
	*(Grey!60!white)2 & *(Grey!60!white)4 & \none & \none\\
	*(Grey!60!white)1 & *(Grey!60!white)1 & *(Grey!60!white)2 & *(Grey!60!white)3
	\end{ytableau},\
	\begin{ytableau}
	*(Grey!60!white)3 & \none & \none & \none\\
	*(Grey!60!white)2 & *(Grey!60!white)3 & \none & \none\\
	*(Grey!60!white)1 & *(Grey!60!white)1 & *(Grey!60!white)3 & *(Grey!60!white)4
	\end{ytableau}&&.\
	\end{align*}
	\caption{Process of obtaining the $l$-two-rowed array 
${w=\protect\begin{psmallmatrix}
			1 & 1 & 2 & 3 & 3 & 3 & 4\\ 3 & 4 & 2 & 1 & 1 & 2 & 3
			\protect\end{psmallmatrix}}$ through the \lps\ reverse insertion of the
		pair of \lps\ tableaux $(\pPl(w),\qQl(w))$.}
	\label{figure:ext_reading_tableaux}
\end{figure}

As in the words case, there are pairs of \ps\ tableaux $(P,Q)$ such that
the word array, obtained from the \ps\ reverse insertion applied
to $(P,Q)$, does not insert to $(P,Q)$
under Algorithm~\ref{alg:ArrayExtendedPS}.
 For instance, the reverse \lps\ insertion of the pair of \lps\ tableaux
\[\left(\begin{ytableau}
3 & \none \\
2 & \none \\
1 & 1
\end{ytableau},\
\begin{ytableau}
3 & \none \\
2 & \none\\
1 & 1
\end{ytableau}\right)\]
leads to the two-rowed array $ w =
\begin{psmallmatrix} 1 & 1 & 2 & 3 \\ 3 & 1 & 2 & 1 \end{psmallmatrix}$, which
under the \lps\ version of Algorithm~\ref{alg:ArrayExtendedPS} inserts to the
pair
\[\left(\pPl(w),\qQl(w)\right)=\left(\begin{ytableau}
3 & 2 \\
1 & 1
\end{ytableau},\
\begin{ytableau}
1 & 3\\
1 & 2
\end{ytableau}\right).\]
In fact, from this example we conclude that there are pairs of \ps\ tableaux
such that when applying the \ps\ reverse insertion method we arrive to a word
array that is not a \ps\ word array. So, just as in the words
case, we will have to restrict the pairs of \ps\ tableaux that we can
consider.

We will need auxiliary results in order to prove the existence of a
Robinson--Schensted--Knuth-like correspondence for \ps\ tableaux.

\begin{lem}\label{lem:Std_Q}
	Given $\smx\in\{\sml,\smr\}$, for any $x$-two-rowed array 
$\begin{psmallmatrix} u \\ v
	\end{psmallmatrix}$ we have 
\[\Std_{\smx}\big(\qQ_\xps\big(\begin{psmallmatrix} u \\ v
	\end{psmallmatrix}\big)\big)=\qQ_\xps\big(\begin{psmallmatrix} \stdl(u) \\ v
	\end{psmallmatrix}\big).\] 
\end{lem}
\begin{proof}
	Let $\smx\in\{\sml,\smr\}$ and let $w=\begin{psmallmatrix} u \\ v 
\end{psmallmatrix}=
	\begin{psmallmatrix} u_1 & u_2 & \cdots & u_k\\ v_1 & v_2 & \cdots &
	v_k \end{psmallmatrix}$ be an arbitrary $\smx$-two-rowed array. Then 
$u=x_1^{i_1}
	\cdots x_n^{i_n}$ for some symbols $x_1<\dots<x_n$ and indexes 
$i_1,\ldots,
	i_n\in\mathbb{N}$ such that $i_1+\cdots+i_n=k$. Thus,
	\begin{equation*}
	w=
	\begin{pmatrix} x_1 & \cdots & x_1 & x_2 & \cdots & x_2 & \cdots & x_{n} \\
	v_1 & \cdots & v_{i_1} & v_{i_1+1} & \cdots & v_{i_1+i_2} & \cdots & v_k
	\end{pmatrix}.
	\end{equation*}
	Since $u=x_1^{i_1}\cdots x_n^{i_n}$, then $\stdl(u)=\left(x_1\right)_1
	\cdots \left(x_1\right)_{i_1}
	\cdots \left(x_n\right)_{1}\cdots \left(x_n\right)_{i_n}$.
	Let $i_0=0$. In the \lps\ case
	(resp., \rps), since $w$ is a $l$-two-rowed word, for any 
$j\in\{0,\ldots, n-1\}$,  $m\in\{1, \ldots,i_{j+1}-1\}$ we have	
	\[v_{i_1+\cdots+i_j+m}\leq 
v_{i_1+\cdots+
		i_j+m+1}\] (resp., $\geq$). By  
	Algorithm~\ref{alg:ArrayExtendedPS}, it follows that for any 
$k\in\{1,\ldots, n\}$, if $\left(x_k\right)_p
	<\left(x_k\right)_q$ for some  $p,q\in\{1,\ldots,i_k
	\}$ with $p<q$, then in $\qQl\big(\begin{psmallmatrix} \stdl(u) \\ 
v
	\end{psmallmatrix}\big)$ (resp., $\qQr\big(\begin{psmallmatrix} 
\stdl(u) \\ v
	\end{psmallmatrix}\big)$), the symbol
	$\left(x_k\right)_q$ is in a column to the right of the column
	containing the  symbol $\left(x_k\right)_p$ (resp., either in the same 
column of $\left(x_k\right)_p$ on top
	of it or in a column to the left of it). 
	
	Since $\qQl(w)$ (resp., $\qQr(w)$) is
	an \lps\ (resp., \rps) tableau, by the process of left (resp., right) 
standardization of
	a tableau, it follows that for any $k\in\{1,\ldots,n\}$, if 
$\left(x_k\right)_p<\left(x_k\right)_q$ for
	some  $p,q\in\{1,\ldots,i_k\}$ with $p<q$, then in
	$\qQl\big(\begin{psmallmatrix} \stdl(u) \\ v
	\end{psmallmatrix}\big)$ (resp., $\qQr\big(\begin{psmallmatrix} 
\stdl(u) \\ v
	\end{psmallmatrix}\big)$), the symbol $\left(x_k\right)_q$ is
	 in a column to the right of the column containing the symbol
	$\left(x_k\right)_p$ (resp., either in the same column of 
$\left(x_k\right)_p$ on top of it or in a 
column to
	the left of it).  
	
	As both $u$ and $\stdl(u)$ are inserted
	according to the \lps\ (resp., \rps) insertion of $v$, it
	follows that the underlying indexes of symbols in the same
	position of $\qQl(w)$ and $\qQl\big(\begin{psmallmatrix} \stdl(u) \\ v
	\end{psmallmatrix}\big)$ (resp., $\qQr(w)$ and 
$\qQr\big(\begin{psmallmatrix} \stdl(u) \\ v
	\end{psmallmatrix}\big)$) are the same. This together with the previous considerations allow us
	to conclude that $\Std_{\sml}(\qQl(w))=\qQl\big(\begin{psmallmatrix} 
\stdl(u) \\ v
	\end{psmallmatrix}\big)$ (resp., $\Std_{\smr}(\qQr(w))
	=\qQr\big(\begin{psmallmatrix} \stdl(u) \\ v
	\end{psmallmatrix}\big)$).
\end{proof}

\begin{prop}\label{prop:standardization_P_Q}
	For $\smx\in\{\sml,\smr\}$ and any $\smx$-two-rowed array 
$w=\begin{psmallmatrix} u \\ v
	\end{psmallmatrix}$ we have:
	\begin{itemize}
		\item[$(i)$] 
$\Std_{\smx}(\qQ_\xps(w))=\qQ_\xps\big(\begin{psmallmatrix} \stdl(u) \\ v
		\end{psmallmatrix}\big)=\qQ_{\xps}\Big(\begin{psmallmatrix} \stdl(u) \\ \std_\smx(v)
		\end{psmallmatrix}\Big)$;
		\item[$(ii)$] $\Std_{\smx}\left(\pP_\xps(w)\right)=\pP_\xps\Big(\begin{psmallmatrix} \stdl(u) \\ \std_\smx(v)
		\end{psmallmatrix}\Big)=\pP_\xps\Big(\begin{psmallmatrix} u \\ \std_\smx(v)
		\end{psmallmatrix}\Big)$.
	\end{itemize}
\end{prop}

\begin{proof}
	The first equality of $(i)$ is just Lemma~\ref{lem:Std_Q}. The second 
equality
	follows from the fact that the first component of the pair obtained from the
	extended $\smx$\ps\ insertion of the $\smx$-two-rowed array 
$\begin{psmallmatrix} u \\ v
	\end{psmallmatrix}$ under Algorithm~\ref{alg:ArrayExtendedPS} is equal to the
	insertion of $v$ under Algorithm~\ref{alg:PSinsertion} together with the fact that
	$v$ and $\std_\smx(v)$ have equivalent $\smx$\ps\ insertions, by 
Lemma~\ref{lemma:equivalent_insertions}.
	
	Regarding $(ii)$, note that the first component of the pair obtained from the
	extended $\smx$\ps\ insertion of the $\smx$-two-rowed array 
$\begin{psmallmatrix} u \\ v
	\end{psmallmatrix}$ under Algorithm~\ref{alg:ArrayExtendedPS} is equal to the
	insertion of $v$ under Algorithm~\ref{alg:PSinsertion}. Thus from 
Remark~\ref{rmk:std_insertion} it follows that 
	$\Std_{\smx}\left(\pP_\xps(\begin{psmallmatrix} u \\ v
	\end{psmallmatrix})\right)=\pP_{\xps}\Big(\begin{psmallmatrix} u \\ \std_\smx(v)
	\end{psmallmatrix}\Big)=\pP_{\xps}\Big(\begin{psmallmatrix} \stdl(u) \\ \std_\smx(v)
	\end{psmallmatrix}\Big)$.
\end{proof}

If we relax the definition of standard stable pairs set over an alphabet so that
it allows the second component of the pair to be a standard \ps\ tableau over that
alphabet instead of just a recording tableau, then the bijection of
Proposition~\ref{prop:std_rob_schensted} is given by the mapping
\[w=\begin{psmallmatrix} u \\ v \end{psmallmatrix}\longmapsto \left(\pP(w),\qQ(w)\right),\]
where $\begin{psmallmatrix} u \\ v \end{psmallmatrix}$ is a standard two-rowed 
array, that is, $u=u_1\cdots u_k$ and $v$ are standard words of the same length 
such that $u_1<\dots<u_k$ and
$\left(\pP(w),\qQ(w)\right)$ is the insertion of $w$ under any of the versions of Algorithm~\ref{alg:ArrayExtendedPS}.

With this new definition of standard stable pairs set, for
$\smx\in\{\sml,\smr\}$, let the \emph{semistandard $\smx$\ps\ stable pairs set
over an alphabet $\bB$} be the set composed by the pairs of $\smx$\ps\ tableaux $(P,Q)$
of the same shape over $\bB$ such that $\big(\Std_{\smx}(P), 
\Std_{\smx}(Q)\big)$ is
in the stable pairs set over $\mathcal{C}(\bB)$.

\begin{thm}\label{thm:array_RSK}
	For $\smx\in\{\sml,\smr\}$, there is a one-to-one correspondence between
	$\smx$-two-rowed arrays over an alphabet and the pairs of $\smx$\ps\ 
tableaux
	from the semistandard $\smx$\ps\ stable pairs set over that alphabet. The
	one-to-one correspondence is given under the mapping
	\[w \mapsto (\pP_\xps(w),\qQ_\xps (w)).\]
\end{thm}
\begin{proof}
	Let $\bB$ be an alphabet and $\smx\in\{\sml,\smr\}$. Let $
	\begin{psmallmatrix} u \\ v \end{psmallmatrix}$ be	a 
$\smx$-two-rowed array over
	$\bB$, and thus $w=\begin{psmallmatrix} \stdl(u) \\ \std_\smx(v) 
\end{psmallmatrix}$ is a standard two-rowed array over $\mathcal{C}(\bB)$.
Therefore, using (the 
two-rowed version of) Proposition~\ref{prop:std_rob_schensted}, 
$\left(\pP_\xps\Big(\begin{psmallmatrix} \stdl(u) \\ \std_\smx(v)
	\end{psmallmatrix}\Big),\qQ_\xps\Big(\begin{psmallmatrix} \stdl(u) \\ \std_\smx(v)
	\end{psmallmatrix}\Big)\right)$ is a pair in 
	the standard stable pairs set
	over $\mathcal{C}(\bB)$. By Proposition~\ref{prop:standardization_P_Q}~$(ii)$,
	it follows that $\pP_\xps\Big(\begin{psmallmatrix} \stdl(u) \\ 
\std_\smx(v) 
\end{psmallmatrix}\Big)=\Std_{\smx}\left(\pP_\xps\left(\begin{psmallmatrix} u \\ 
v
	\end{psmallmatrix}\right)\right)$, and by $(i)$ of the same proposition 
it follows that
	$\qQ_\xps\Big(\begin{psmallmatrix} \stdl(u) \\ \std_\smx(v)
\end{psmallmatrix}\Big)=\Std_{\smx}\left(\qQ_\xps\left(\begin{psmallmatrix} u 
\\ v
	\end{psmallmatrix}\right)\right)$. So, we conclude that from the 
two-rowed array $\begin{psmallmatrix} u \\ v \end{psmallmatrix}$ we obtain a 
pair $\left(\pP_\xps\Big(\begin{psmallmatrix} u \\ v
	\end{psmallmatrix}\Big),\qQ_\xps\Big(\begin{psmallmatrix} u \\ v
	\end{psmallmatrix}\Big)\right)$ that is in the semistandard $\smx$\ps\ stable pairs set
	over $\bB$. Note that different two-rowed arrays lead to different 
pairs.
	
	The other way around, let $(R,S)$ be a pair of $\smx$\ps\ tableaux in the semistandard
	$\smx$\ps\ stable pairs set over $\bB$. Then, by definition
	$\left(\Std_{\smx}(R),\Std_{\smx}(S)\right)$ is in the standard stable 
pairs set
	over $\mathcal{C}(\bB)$. So, by
	(the two-rowed version of) Proposition~\ref{prop:std_rob_schensted}, 
there
	is a standard two-rowed array $\begin{psmallmatrix} u' \\ v' 
\end{psmallmatrix}$ over
	$\mathcal{C}(\bB)$ that maps to this pair.
	
	The word $u'$ is ordered strictly
	increasingly and so $u'=\stdl(u)$, for a word $u$ over $\bB$, obtained 
from $u'$ removing the indexes, being $u$ also ordered weakly 	increasingly. 

Also $v'=\std_\smx(v)$, for some word $v$ over $\bB$. Indeed, 
the first component of
	the pair obtained from the extended $\smx$\ps\ insertion of the
$\smx$-two-rowed
	array $\begin{psmallmatrix} u' \\ v' \end{psmallmatrix}$ under 
Algorithm~\ref{alg:ArrayExtendedPS}, is equal to the
	insertion of $v'$ under Algorithm~\ref{alg:PSinsertion}, so
	$\pP_\xps(v')=\Std_{\smx}(R)$. Thus, $v'$ has the claimed form, by 
Lemma~\ref{rmk:std_surjection}. 

Hence, 
$\left(\pP_\xps\Big(\begin{psmallmatrix} \stdl(u) \\ \std_\smx(v)
	\end{psmallmatrix}\Big),\qQ_\xps\Big(\begin{psmallmatrix} \stdl(u) \\ 
\std_\smx(v)  
\end{psmallmatrix}\Big)\right)=\left(\Std_{\smx}(R),\Std_{\smx}(S)\right)$. To 
see that $\begin{psmallmatrix} u \\ v \end{psmallmatrix}$ is a $\smx$-two-rowed 
array over $\bB$, it remains to check that condition $(2)$ holds in  case 
$\smx=l$, and condition $(3)$ holds in case $\smx=r$. In case $\smx=l$, if 
condition $(2)$ did not hold, then there would exist symbols $a_i,a_j$ from 
$u'$, with $i<j$ (and thus $a_i<a_j$), such that $a_j$ would be inserted either 
in the same column of $a_i$ on top of it, or in a column to the left of $a_i$. 
This contradicts $\qQ_\xps\Big(\begin{psmallmatrix} \stdl(u) \\ \std_\smx(v)
	\end{psmallmatrix}\Big)=\Std_{\smx}(S)$. With a similar reasoning 
we deduce the case $\smx=r$.
	
	Finally, by Proposition~\ref{prop:standardization_P_Q}, we get
$\Std_{\smx}\left(\pP_\xps(w)\right)= \Std_{\smx}(R)$ and 
$\Std_{\smx}(\qQ_\xps(w))=\Std_{\smx}(S)$, for $w=\begin{psmallmatrix} u \\ v 
\end{psmallmatrix}$. Thus, destandardizing the $\smx$\ps\ tableaux we deduce 
that	 the two-rowed array $w$ over $\bB$  maps to $(R,S)$.  
The result follows.
\end{proof}

In the following paragraphs we illustrate Theorem~\ref{thm:array_RSK} using the 
\lps\ reverse insertion.
Consider the following pair of \lps\ tableaux
\[\left(R,S\right)=\left(\begin{ytableau}
\none & 2 & \none & 3 & \none & \none\\
1 & 1 & 2 & 2 & 3 & 3
\end{ytableau},\
\begin{ytableau}
\none & 2 & \none & 3 & \none & \none\\
1 & 1 & 1 & 2 & 4 & 4
\end{ytableau}\right)\]
in the semistandard \lps\ stable pairs set over $\aA_4$.

By the previous theorem, $\left(R,S\right)=\left(\pPl(\begin{psmallmatrix} u \\ 
v \end{psmallmatrix}),\qQl(\begin{psmallmatrix} u \\ v 
\end{psmallmatrix})\right)$
for some $l$-two-rowed array $\begin{psmallmatrix} u \\ v \end{psmallmatrix}$. 

Applying left standardization to the pair of tableaux $(R,S)$
we obtain
\begin{align*}
  &\left(\Std_{\sml}(R), 
\Std_{\sml}(S)\right)=\left(\begin{ytableau}
\none & 2_1 & \none & 3_1 & \none & \none\\
1_1 & 1_2 & 2_2 & 2_3 & 3_2 & 3_3
\end{ytableau},\
\begin{ytableau}
\none & 2_1 & \none & 3_1 & \none & \none\\
1_1 & 1_2 & 1_3 & 2_2 & 4_1 & 4_2
\end{ytableau}\right).
\end{align*}
Then, applying the \lps\ reverse insertion to this pair, we obtain the 
$l$-two-rowed array
\begin{align*}
\begin{pmatrix} \stdl(u) \\ \stdl(v) \end{pmatrix}=
\begin{pmatrix} 1_1 & 1_2 & 1_3 & 2_1 & 2_2 & 3_1 & 4_1 & 4_2\\
1_1 & 2_1 & 2_2 & 1_2 & 3_1 & 2_3 & 3_2 & 3_3 \end{pmatrix}
\end{align*}
Applying destandardization to both $\stdl(u)$ and $\stdl(v)$, we obtain
\begin{align*}
\begin{pmatrix} u \\ v \end{pmatrix}=
\begin{pmatrix} 1 & 1 & 1 & 2 & 2 & 3 & 4 & 4\\
1 & 2 & 2 & 1 & 3 & 2 & 3 & 3 \end{pmatrix},
\end{align*}
which can be verified to insert to the pair $\left(P,Q\right)$ under
Algorithm~\ref{alg:ArrayExtendedPS}.

\section{Bell numbers and the number of \ps\ tableaux}
\label{sec:bellnumbers}

The $n$-th Bell number counts the number of different partitions of a set having $n$ distinct elements (sequence
A000110 in the OEIS).

Given $n\in \mathbb{N}$, as noticed in \cite{BL2005}, if we represent set
partitions of the totally ordered set $\aA_n$ by:
\begin{itemize}
	\item ordering decreasingly the sets that compose set partitions;
	\item ordering set partitions increasingly according to the minimum
	elements of their sets;
\end{itemize}
we establish a one-to-one correspondence between the partitions of the set
$\aA_n$ and the \ps\ tableaux over standard words of
$\mathfrak{S}(\aA_n)$. So, if we denote the number of \ps\ tableaux
over standard words of $\mathfrak{S}(\aA_n)$ by
\begin{equation*}
P(1,\ldots,1)=\abs{\{\pP(\sigma):\sigma\in\mathfrak{S}(\aA_n)\}},
\end{equation*}
where $(1,\ldots,1)$ is a sequence with $n$ symbols $1$ and is equal to the
evaluation of any standard word $\sigma\in \mathfrak{S}(\aA_n)$, then $P(1,
\ldots,1)$ is given by the $n$-th Bell number, $B_n$.

The Stirling number of second kind \cite[Chapter~4]{MR1411676}, denoted $S(n,k)$, is
the number of ways of partitioning $n$ different elements into $k$ distinct
sets ($k$ columns, in our setting). So, $S(n,k)$ gives the number of \ps\
tableaux over words of $\mathfrak{S}(\aA_n)$ that have exactly $k$
columns. As known, the $n$-th Bell number is given by the sum of the
Stirling numbers $S(n,k)$, where $k$ ranges over the set $\{1,\ldots, n\}$.

The goal in the remainder of this section is to follow a similar approach
for the case where words are taken over $\aA_n^*$. Our objective is to count
both the number of \lps\ and \rps\ tableaux over words of $\aA^*_n$
with a given fixed evaluation.


\subsection{Counting \ps\ tableaux over $\aA^*_n$}
In this subsection we present formulas to count both the number of \lps\ and
\rps\ tableaux over words of $\aA_n$ with a given fixed evaluation
$(m_1,\ldots,m_n)$, that we shall denote by $L(m_1,\ldots,m_n)$ and $R(m_1,
\ldots,m_n)$, respectively. In particular, we consider the case when each
symbol of $\aA_n$ occurs exactly once. That is, we consider the case
of standard words of $\mathfrak{S}(\aA_n)$, where we count $L(1,\ldots,1)$ and
$R(1,\ldots,1)$, respectively. These quantities are both equal to $P(1,\ldots,1)$
which, as we noted, is equal to the $n$-th Bell number, $B_n$.

Before we proceed, we observe that in
general, $L(m_1,m_2, \ldots, m_n)\leq P(1,\ldots,1)$ and 
$R(m_1,m_2,\ldots,m_n)\leq P(1,\ldots,1)$, for any sequence
$(1,\ldots,1)$ with $m_1+m_2+\dots+m_n$ symbols $1$.


Notice also that, the numbers $L(m_1,m_2, \ldots, m_n)$ and $R(m_1,m_2, \ldots, 
m_n)$ only depend on the non-zero entries of the evaluation sequence 
$(m_1,\ldots,m_n)$. Indeed, for any
$n\in \mathbb{N}$, the number of $\smx$\ps\ tableaux with evaluation
$(m_1, \ldots,m_{k-1},0,m_{k+1},\ldots, m_{n+1})$ is given by the number of
$x$ tableaux with evaluation $(m_1,\ldots,m_{k-1},m_{k+1},
\ldots,m_{n+1})$.   By convenience we assume that $(m_1,\ldots,m_n)\in
\mathbb{N}^n$.


It is easy to see that the length of the bottom row of an \rps\ tableaux with 
evaluation  $(m_1,m_2,\ldots,m_n)$ is greater or equal than
	$1$ and smaller or equal than $n$. As for the $L$ case we have the next 
result.

\begin{lem}
	The length of the bottom row of an \lps\ tableaux with evaluation 
$(m_1,m_2,\ldots,m_n)$ is greater or equal than $\max\{m_1,m_2,\ldots,m_n\}$ and 
less or equal than $m_1+m_2+\dots +m_n$. 
\end{lem}
\begin{proof}
	Let $R$ be an \lps\ tableau 
	with evaluation $(m_1,m_2,\ldots,m_n)$.
	The $m_1$ symbols $1$ of $R$ will all appear in the bottom row. Also, as any
	column of $R$ can contain no more than one occurrence of each
	symbol of $\aA_n$, the number of columns of $R$ is greater or equal than
	$\max\{m_1,\ldots,m_n\}$. The \lps\ insertion of the weakly
	increasing word with evaluation $(m_1,\dots,m_n)$ produces a tableau
	with $m_1+\dots +m_n$ columns. Thus any \lps\ tableau with evaluation
	$(m_1,\dots,m_n)$ has at most $m_1+\dots 
+m_n$ columns.
%
\end{proof}

Throughout the remainder of this text, we also assume the
convention that for $k>m$ the binomial coefficient $\binom{m}{k}$ is $0$.

The idea is to start by presenting formulas to count both the number of \lps\
tableaux and the number of \rps\ tableaux with a given fixed
evaluation and a fixed bottom row. So, given sequences $(m_1,m_2,\ldots,m_n)\in
\mathbb{N}^n$ and $(j_2,j_3,\ldots,j_n)\in \mathbb{N}_0^{n-1}$, let
\begin{align*}
\begin{bmatrix}
m_1 &  m_2 & \cdots & m_n \\
0   &  j_2 & \cdots & j_n
\end{bmatrix}^\bell=\binom{m_1}{m_2-j_2} \cdots \binom{m_1+ j_2 + \cdots +
	j_{n-1}}{m_n-j_n}
\end{align*}

\begin{lem}\label{prop:number_lps_tableaux_fixed_content}
	The number of \lps\ tableaux having evaluation $(m_1,\ldots,m_n)\in
	\mathbb{N}^n$ whose \lps\ bottom row has evaluation \allowbreak $(m_1,
	 j_2,\ldots, j_n)\in\mathbb{N}\times\mathbb{N}_0^{n-1}$,
	is given by the formula
	\[  \begin{bmatrix}
	m_1 &  m_2 & \cdots & m_n \\
	0   &  j_2 & \cdots & j_n
	\end{bmatrix}^\bell.\]
\end{lem}
\begin{proof}
	These tableaux have $j_2$ symbols $2$ in the bottom row and since any
	\lps\ column cannot contain more than one symbol from $\aA_n$, by the order of
	\lps\ columns and \lps\ bottom rows, the remaining $m_2-j_2$ symbols $2$ are in
	the columns (in the second row) that contain the symbol $1$. Thus, $m_2-j_2$ is
	less or equal than $m_1$ and there are $\binom{m_1}{m_2-j_2}$ possibilities to
	place those $2$'s.

	More generally, such tableaux will have $j_a$ symbols $a$ in the
	bottom row, and so the remaining $m_a-j_a$ symbols $a$ cannot be inserted in
	columns to the right of the columns containing symbols $a$. Therefore, they
	must be placed in columns that have symbols from $\aA_{a-1}$. There are $m_1+
	j_2+\cdots + j_{a-1}$ such columns.
	Thus $m_a-j_a$ is less or equal than $m_1+j_2+\cdots+j_{a-1}$ and there are $\binom{m_1+j_2+\cdots+j_{a-1}}{m_a-j_a}$ possibilities to place those symbols.
	The result follows.
\end{proof}

From the previous proof we deduce that for \lps\ tableaux with
evaluation $(m_1,\ldots,m_n)$ whose bottom row has evaluation 
$(m_1,j_2,\ldots,
j_n)$, for any $a\in\aA_n \setminus \{1\}$, we have $0\leq m_a-j_a\leq m_1+\sum_{b=2}^{a-1}
j_b$.


Regarding the \rps\ case, for any sequences $(m_1,m_2,\ldots,m_n)\in \mathbb{
	N}^n$ and $(j_1,j_2,\ldots, j_n)\in \{0,1\}^{n}$, let
\begin{align*}
\begin{bmatrix}
&  m_2 & \cdots & m_n \\
j_1   &  j_2 & \cdots & j_n
\end{bmatrix}^\belr=\binom{m_2+j_1}{m_2-j_2} \cdots \binom{m_n+ (j_1+j_2 +
	\cdots+j_{n-1})}{m_n-j_n}.
\end{align*}

\begin{lem}\label{prop:number_rps_tableaux_fixed_content}
	The number of \rps\ tableaux with evaluation
	$(m_1,\ldots,m_n)\in\mathbb{N}^n$ and whose \rps\ bottom row has evaluation
	the $0$-$1$ sequence $(1,j_2,\ldots, j_n)$, is given by the formula
	\begin{align*}
	\begin{bmatrix}
	& m_2 & \cdots & m_n \\
	0 &  j_2 & \cdots & j_n
	\end{bmatrix}^\belr.
	\end{align*}
\end{lem}
\begin{proof}
	First note that the number of ways to distribute $k$ symbols into $l$ columns
	can be calculated using the ``stars and bars'' method. That is, the $k$
	symbols can be viewed as $k$ stars which we want to separate into $l$ groups
	by using $l-1$ bars in between. So, this number is given by the number of
	ways of choosing $l-1$ positions from $k+l-1$ spaces, that is, $\binom{k+l-1}{l-1}=\binom{k+l-1}{k}$.

	All the symbols $1$ occur in the bottom part of the first column of these
	tableaux. Moreover, the first column is followed by $j_2\in\{0,1\}$ 
columns
	whose bottom symbol is $2$. Considering the order of \rps\ columns and \rps\
	bottom rows, the remaining $m_2-j_2$ symbols $2$ have to be distributed over
	their first $j_1+j_2$ columns. By the first paragraph there are
	\begin{equation*}
	\binom{m_2-j_2+ j_1+j_2-1}{m_2-j_2}=
	\binom{m_2}{m_2-j_2}
	\end{equation*}
	such possibilities.
	
	In general, these tableaux have $j_a\in\{0,1\}$ columns whose bottom 
symbol
	is $a$ and thus the remaining $m_a-j_a$ symbols have to be inserted over the
	first  $j_1+j_2+\cdots+j_a$ columns. By the first paragraph, there are
	\begin{align*}
	\binom{m_a-j_a+(j_1+j_2+\cdots+j_a)-1}{m_a-j_a}=
	\binom{m_a+(j_2+\cdots+j_{a-1})}{m_a-j_a}
	\end{align*}
	possibilities to place those symbols. The result follows.
\end{proof}

\begin{lem}\label{prop:invariance_rps}
	The number $R(m_1,\ldots,m_n)$ of \rps\ tableaux with evaluation $(m_1,\ldots,m_n)\in \mathbb{N}^n$ is independent of the choice of $m_1$.
	Also, both the numbers $L(m)$ and $R(m)$ of \lps\ and \rps\
	tableaux with evaluation $(m)$ is equal to 1.
\end{lem}
\begin{proof}
	The first part of the proposition follows from the fact that all the
	$m_1$ symbols $1$ are in the bottom part of the first column and
	therefore $j_1=1$ for any $m_1\in \mathbb{N}$. The second part of the
	proposition is immediate.
\end{proof}

\begin{prop}\label{prop:number_Bell_classes}
	The numbers $L(m_1,\ldots,m_n)$ and $R(m_1,\ldots,m_n)$ of, respectively,
	\lps\ tableaux and \rps\ tableaux with evaluation
	$(m_1,\ldots,m_n)$, are given by
	\begin{align*}
	&L(m_1,\ldots,m_n)=\displaystyle \sum_{(0,\ldots,0)\leq (j_2,\ldots,
		j_n)\leq (m_2,\ldots,m_n)}
	\begin{bmatrix}
	m_1 &  m_2 & \cdots & m_n \\
	0   &  j_2 & \cdots & j_n
	\end{bmatrix}^\bell
	\end{align*}
	and
	\begin{align*}
	R(m_1,\ldots,m_n)=\displaystyle \sum_{(0,\ldots,0)\leq (j_2,\ldots,
		j_n)\leq (1,\ldots,1)}
	\begin{bmatrix}
	& m_2 & \cdots & m_n \\
	0 &  j_2 & \cdots & j_n
	\end{bmatrix}^\belr.
	\end{align*}
\end{prop}

\begin{prop}\label{prop:recur_lps_rps}
	The numbers $L(m_1,\ldots,m_n)$ and $R(m_1,\ldots,m_n)$ of,
	respectively, \lps\ and \rps\ tableaux with
	evaluation $(m_1,\ldots, m_n)$ are recursively obtained in the
	following way
	\begin{equation*}
	L(m_1,m_2,\ldots,m_{n})=
	\sum_{0 \leq j_2\leq m_2}\binom{m_1}{m_2-j_2}L(m_1+j_2,m_3,\ldots,
	m_{n}).
	\end{equation*}
	and
	\begin{align*}
	R(m_1,m_2,\ldots,m_{n})&= R(m_2,m_3,\ldots,m_n)+\\
	&+\sum_{(0,\ldots,0) \leq (j_3,\ldots,j_n)\leq (1,\ldots,1)}m_2
	\begin{bmatrix}
	& m_3 & \cdots & m_n \\
	1 &  j_3 & \cdots & j_n
	\end{bmatrix}^\belr,
	\end{align*}
	where $L(m)=1=R(m)$ and $R(m,n)=\binom{n}{n}+\binom{n}{n-1}=1+n$,
	for all $m,n\in \mathbb{N}$.
\end{prop}
\begin{proof}
	Note that
	\begin{equation*}
	\begin{bmatrix}
	m_1 &  m_2 & \cdots & m_n \\
	0   &  j_2 & \cdots & j_n
	\end{bmatrix}^\bell=\binom{m_1}{m_2-j_2}\begin{bmatrix}
	m_1+j_2 &  m_3 & \cdots & m_n \\
	0   &  j_3 & \cdots & j_n
	\end{bmatrix}^\bell
	\end{equation*}
	and for $j\in\mathbb{N}_0$
	\begin{equation*}
	\begin{bmatrix}
	& m_2 & \cdots & m_n \\
	j &  j_2 & \cdots & j_n
	\end{bmatrix}^\belr=\binom{m_2+j}{m_2-j_2}\begin{bmatrix}
	& m_3 & \cdots & m_n \\
	j+j_2 &  j_3 & \cdots & j_n
	\end{bmatrix}^\belr.
	\end{equation*}
	The result now follows from Lemma~\ref{prop:invariance_rps}
	and Proposition~\ref{prop:number_Bell_classes}, observing that from
	Proposition~\ref{prop:number_Bell_classes} we have
	\[R(m_1,m_2)=\displaystyle \sum_{0\leq j_2\leq 1}
	\begin{bmatrix}
	& m_2\\
	0 &  j_2
	\end{bmatrix}^\belr=\binom{m_2}{m_2}+\binom{m_2}{m_2-1}=1+m_2.\]
\end{proof}

As an example, we use the above formula to compute the number of \lps\ tableaux
with evaluation $(2,1,2)$
\begin{align*}
L(2,1,2)&=\sum_{j_2=0}^{1}\binom{2}{1-j_2}L(2+j_2,2)=\binom{2}{1}\cdot L(2,2)+\binom{2}{0}\cdot L(3,2)\\
&=2\cdot\left(\binom{2}{2}+\binom{2}{1}+\binom{2}{0}\right)+
1\cdot\left(\binom{3}{2}+\binom{3}{1}+\binom{3}{0}\right)\\
& =7+8=15
\end{align*}
Indeed,  $\binom{2}{0}L(3,2)$ is obtained when $j_2=1$ and therefore is
the number of such \lps\ tableaux with the symbol
$2$ on the bottom row:
\begin{alignat*}{5}
\ytableausetup
{mathmode, boxsize=1em, aligntableaux=center}
&j_3=2\ \ &\rightarrow \ \ &{\color{gray}\binom{3}{0}} \ \ &\rightarrow \ \
&\begin{ytableau}
1 & 1 & 2 & 3 & 3
\end{ytableau}\\
&j_3=1\ \ &\rightarrow \ \ &{\color{gray}\binom{3}{1}} \ \ &\rightarrow \ \
&\begin{ytableau}
3\\
1 & 1 & 2 & 3
\end{ytableau}\, ,\
\begin{ytableau}
\none & 3\\
1 & 1 & 2 & 3
\end{ytableau}\, ,\
\begin{ytableau}
\none & \none & 3\\
1 & 1 & 2 & 3
\end{ytableau}\\
&j_3=0\ \ &\rightarrow \ \ &{\color{gray}\binom{3}{2}} \ \ &\rightarrow \ \
&\begin{ytableau}
3 & 3\\
1 & 1 & 2
\end{ytableau}\, ,\
\begin{ytableau}
3 & \none & 3\\
1 & 1 & 2
\end{ytableau}\, ,\
\begin{ytableau}
\none & 3 & 3\\
1 & 1 & 2
\end{ytableau}
\end{alignat*}
The number $\binom{2}{1}L(3,2)$ counts those \lps\ tableaux where the
symbol $2$ is on top of the symbols $1$ on the second row:
\begin{alignat*}{5}
\ytableausetup
{mathmode, boxsize=1em, aligntableaux=center}
&j_3=2\ \ &\rightarrow \ \ &{\color{gray}\binom{2}{1}\cdot\binom{2}{0}} \ \
&\rightarrow \ \ &\begin{ytableau}
2\\
1 & 1 & 3 & 3
\end{ytableau}\, ,\
\begin{ytableau}
\none & 2\\
1 & 1 & 3 & 3
\end{ytableau}\\
&j_3=1\ \ &\rightarrow \ \ &{\color{gray}\binom{2}{1}\cdot\binom{2}{1}} \ \
&\rightarrow \ \ &\begin{ytableau}
3\\
2\\
1 & 1 & 3
\end{ytableau}\, ,\
\begin{ytableau}
2 & 3\\
1 & 1 & 3
\end{ytableau}\, ,\
\begin{ytableau}
3 & 2\\
1 & 1 & 3
\end{ytableau}\, ,\
\begin{ytableau}
\none & 3\\
\none & 2\\
1 & 1 & 3
\end{ytableau}\\
&j_3=0\ \ &\rightarrow \ \ &{\color{gray}\binom{2}{1}\cdot\binom{2}{2}} \ \
&\rightarrow \ \ &\begin{ytableau}
3\\
2 & 3\\
1 & 1
\end{ytableau}\, ,\
\begin{ytableau}
\none & 3\\
3 & 2\\
1 & 1
\end{ytableau}
\end{alignat*}

Analogously, the number of \rps\ tableaux with evaluation $(2,1,2)$ is given by:
\begin{align*}
R(2,1,2)&=R(1,2)+\sum_{j_3=0}^{1}1\begin{bmatrix}
& 2 \\
1 & j_3
\end{bmatrix}^\belr\\
&=\sum_{j_3=0}^{1}\binom{2}{2-j_3}+\left(\begin{bmatrix}
& 2 \\
1 & 0
\end{bmatrix}^\belr+\begin{bmatrix}
& 2 \\
1 & 1
\end{bmatrix}^\belr\right)\\
&=\left(\binom{2}{2}+\binom{2}{1}\right)+
\left(\binom{3}{2}+\binom{3}{1}\right)=3+6=9
\end{align*}
The number $R(1,2)$ counts the number of those \rps\ tableaux when
$j_2=0$, that is, whenever the symbol $2$ is on top  of the symbols $1$ in the
first column:
\begin{alignat*}{5}
\ytableausetup
{mathmode, boxsize=1em, aligntableaux=center}
&j_3=0\ \ &\rightarrow \ \ &{\color{gray}\binom{2}{2}} \ \ &\rightarrow \ \
&\begin{ytableau}
3\\
3\\
2\\
1\\
1
\end{ytableau}\\
&j_3=1\ \ &\rightarrow \ \ &{\color{gray}\binom{2}{1}} \ \ &\rightarrow \ \
&\begin{ytableau}
3\\
2\\
1\\
1 & 3
\end{ytableau}\, ,\
\begin{ytableau}
2\\
1 & 3\\
1 & 3
\end{ytableau}
\end{alignat*}
The number $\displaystyle \sum_{j_3=0}^{1}1\begin{bmatrix}
& 2 \\
1 & j_3
\end{bmatrix}^\belr$ counts those \rps\ tableaux whenever $j_2=1$, that
is, it counts the \rps\ tableaux where $2$ is in the bottom row:
\begin{alignat*}{5}
\ytableausetup
{mathmode, boxsize=1em, aligntableaux=center}
&j_3=0\ \ &\rightarrow \ \ &{\color{gray}\binom{3}{2}} \ \ &\rightarrow \ \
&\begin{ytableau}
3\\
3\\
1\\
1 & 2
\end{ytableau}\, ,\
\begin{ytableau}
3\\
1 & 3\\
1 & 2
\end{ytableau}\, ,\
\begin{ytableau}
\none & 3\\
1 & 3\\
1 & 2
\end{ytableau}\\
&j_3=1\ \ &\rightarrow \ \ &{\color{gray}\binom{3}{1}} \ \ &\rightarrow \ \
&\begin{ytableau}
3\\
1\\
1 & 2 & 3
\end{ytableau}\, ,\
\begin{ytableau}
1 & 3\\
1 & 2 & 3
\end{ytableau}\, ,\
\begin{ytableau}
1 & \none & 3\\
1 & 2 & 3
\end{ytableau}
\end{alignat*}

Applying Proposition~\ref{prop:number_Bell_classes} to the standard case, we
get:
\begin{cor}
	If $(1,\ldots,1)$ is a sequence with $n$ symbols $1$, the $n$-th Bell number,
	$B_n$, is given by both $L(1,\ldots,1)$ and $R(1,\ldots,1)$, which are equal
	to
	\begin{align*}
	\displaystyle
	\sum_{0 \leq p_2,\ldots,p_n\leq 1} (1+p_2)^{(1-p_3)}\cdots (1+p_2+\cdots
	+p_{n-1})^{(1-p_n)}.
	\end{align*}
\end{cor}
\begin{proof}
	Both the numbers $\begin{bmatrix}
	m_1 &  m_2 & \cdots & m_n \\
	0   &  j_2 & \cdots & j_n
	\end{bmatrix}^\bell$ and $\begin{bmatrix}
	&  m_2 & \cdots & m_n \\
	0   &  j_2 & \cdots & j_n
	\end{bmatrix}^\belr$, with $m_1=\ldots =m_n=1$ are given by
	\[ \displaystyle
	\binom{1}{1-j_2} \binom{1+j_2}{1-j_3} \binom{1+j_2+j_3}{1-j_4} \cdots
	\binom{1+j_2+\cdots + j_{n-1}}{1-j_n}.  \]
	Since each $j_a$ is either $0$ or $1$, then $1-j_a$ is either $1$ or $0$ and
	thus $\binom{1+j_2+\cdots+j_a}{1-{j_{a+1}}}= (1+j_2+\cdots+j_a)^{(1-j_{a+1})}$,
	for  $a\in\aA_{n-1}\setminus\{1\}$. The result follows from
	Proposition~\ref{prop:number_Bell_classes}.
\end{proof}

For example, $B_4=\displaystyle \sum_{0\leq p_2,p_3,p_4\leq 1}
(1+p_2)^{(1-p_3)}(1+p_2+p_3)^{(1-p_4)}$. The possible triples of $0$'s and $1$'s
are  $(0,0,0)$, $(1,0,0)$, $(0,1,0)$, $(0,0,1)$, $(0,1,1)$, $(1,0,1)$, $(1,1,0)$
and $(1,1,1)$, which maintaining the order gives the sum $1+2\times 2+2+1+1+2+3
+1=15$.

\section{A hook length formula for standard \ps\ tableaux}
For any partition $\lambda$ of a natural number, the hook-length formula is a
formula that gives the number of standard Young tableaux having shape
$\lambda$. Besides that, the hook-length formula also provides the dimension of
the irreducible representation of the symmetric group associated to $\lambda$.

By the bijectivity of the Robinson correspondence, it follows that the 
hook-length formula provides the number of standard words that insert to a 
specific
standard Young tableaux under Schensted's insertion algorithm.

In this section, our goal is to provide an analogous hook-length formula for 
standard Patience Sorting tableaux. As
shapes of \ps\ tableaux are compositions of natural numbers, we will work with compositions instead of partitions. As a
consequence, we deduce a new formula for the Bell numbers and bounds for both factorial numbers and the number of words
inserting to a specific standard Patience Sorting tableaux.

Consider the following algorithm:

\begin{algorithm}[From pre-tableaux to standard \ps\ tableaux]
	\label{alg:pre_to_ps}
	~\par\nobreak
	\textit{Input:} A pre-tableau $T$ from 
$\tab_\mathcal{B}(\lambda)$, with 
$\lambda=(\lambda_1,\lambda_2,\ldots,\lambda_m)$ and
	$\mathcal{B}\subseteq \aA_n$, for some $n\geq m$.

	\textit{Output:} A standard \ps\ tableau $W(T) \in
	\pstab_\mathcal{B}(\lambda)$.

	\textit{Method:}
	
	Let $T=c_1c_2\cdots c_m$, for pre-columns $c_1,\ldots, c_m$ of height 
$\lambda_1,\ldots,\lambda_m$, respectively.
	\begin{enumerate}
		\item Step 1: Let $(j,k)$ denote the column-row position of the smallest
		symbol from $T$.
		In $T$, exchange the symbol in the column-row position $(j,k)$
		with the symbol in the column-row position $(1,1)$. Then rearrange the
		symbols of the obtained first column in increasing order from bottom to
		top. Denote the obtained tableau by $c_1^{(1)}c_2^{(1)}\cdots
		c_m^{(1)}$;
		\item Step $i$ ($2\leq i\leq m$): Let $(j,k)$ denote the column-row
		position of the smallest symbol from $c_i^{(i-1)}\cdots c_m^{(i-1)}$. In
		$c_i^{(i-1)}\cdots c_m^{(i-1)}$, exchange the symbol in the column-row
		position $(j,k)$ with the symbol in the column-row position $(i,1)$. Then
		rearrange the symbols of the obtained $i$-th column in increasing order
		from bottom to top. Denote the obtained tableau by $c_i^{(i)}\cdots
		c_m^{(i)}$.
	\end{enumerate}
	Output the tableau $W(T)=c_1^{(1)}c_2^{(2)}\cdots c_m^{(m)}$.
\end{algorithm}

Observe that, if $\mathcal{B}\subseteq \aA_n$ for some $n\in\mathbb{N}$ and $\lambda\vDash \lvert{\mathcal{B}}\rvert$, then applying the algorithm
to an arbitrary tableau of $\tab_\mathcal{B}(\lambda)$ always yields a tableau in $\pstab_\mathcal{B}(\lambda)$. Indeed,
if $T=c_1c_2\cdots c_m\in \tab_\mathcal{B}(\lambda)$ and for any $i\in\{1,2, \ldots,m\}$, $x_i$ denotes the bottom
symbol of the column $c_i^{(i)}$ of the tableau $W(T)= c_1^{(1)}c_2^{(2)}\cdots c_m^{(m)}$, then as the symbols from
$W(T)$ are all different and $x_i$ is the smallest symbol from $c_i^{(i-1)}\cdots c_m^{(i-1)}$, we deduce that for all
$i\in \{1,2,\ldots,m-1\}$, $x_i< x_{i+1}$. Also, as for any $i\in\{1,2,\ldots,m\}$, $c_i^{(i)}$ is in increasing order
from bottom to top, we conclude that the obtained tableau $W(T)=c_1^{(1)}c_2^{(2)}\cdots c_m^{(m)}$ belongs to
$\pstab_\mathcal{B}(\lambda)$. Furthermore, for any tableau $T\in \pstab_\mathcal{B} (\lambda)$, we have $W(T)=T$.

For example, considering $\mathcal{B}=\{1,2,4,5,6,7,8,9\}\subseteq \aA_9$, the 
steps that we
obtain from applying the previous algorithm to the pre-tableau
\begin{equation*} \label{exmp12}
\ytableausetup
{aligntableaux=center}
T=c_1c_2c_3c_4=\begin{ytableau}
\none & 4 & \none & \none \\
\none & 5 & 1 & 7 \\
9 & 8 & 6 & 2
\end{ytableau}\in\tab_\mathcal{B}\left((1,3,2,2)\right)
\end{equation*}
are
\begin{align*} \label{steps}
\ytableausetup
{aligntableaux=center}
&\text{Step 1: }T = \begin{ytableau}
\none & 4 & \none & \none \\
\none & 5 & 1 & 7 \\
9 & 8 & 6 & 2
\end{ytableau},\ c_1^{(1)}c_2^{(1)}c_3^{(1)}c_4^{(1)}=\begin{ytableau}
\none & 4 & \none & \none \\
\none & 5 & 9 & 7 \\
1 & 8 & 6 & 2
\end{ytableau}\\
&\text{Step 2: }c_2^{(1)}c_3^{(1)}c_4^{(1)}=\begin{ytableau}
4 & \none & \none \\
5 & 9 & 7 \\
8 & 6 & 2
\end{ytableau},\ c_2^{(2)}c_3^{(2)}c_4^{(2)}=\begin{ytableau}
5 & \none & \none \\
4 & 9 & 7 \\
2 & 6 & 8
\end{ytableau}\\
&\text{Step 3: }c_3^{(2)}c_4^{(2)}=\begin{ytableau}
9 & 7 \\
6 & 8
\end{ytableau},\ c_3^{(3)}c_4^{(3)}=\begin{ytableau}
9 & 7 \\
6 & 8
\end{ytableau}\\
&\text{Step 4: }c_4^{(3)}=\begin{ytableau}
7 \\
8
\end{ytableau},c_4^{(4)}=\begin{ytableau}
8\\
7
\end{ytableau}
\end{align*}
and the algorithm outputs
\begin{align*}
W(T)=c_1^{(1)}c_2^{(2)}c_3^{(3)}c_4^{(4)}=\begin{ytableau}
\none & 5 & \none & \none \\
\none & 4 & 9 & 8 \\
1 & 2 & 6 & 7
\end{ytableau}\in\pstab_\mathcal{B}\left((1,3,2,2)\right).
\end{align*}

These observations allow us to conclude that, for any $n\in\mathbb{N}$, $\mathcal{B}\subseteq \aA_n$ and
$\lambda\vDash \lvert{\mathcal{B}}\rvert$ the map
\begin{align*}
w_{\lambda,\mathcal{B}}:\tab_\mathcal{B}(\lambda)&\rightarrow 
\pstab_\mathcal{B}(\lambda)\\
T&\mapsto W(T)
\end{align*}
is well defined. From the fact that
$W(W(T))=W(T)$ it follows that $w_{\lambda,\mathcal{B}}$ is surjective.

\begin{thm}
	If $\mathcal{B}\subseteq \aA_n$, for some $n\in\mathbb{N}$, $\lambda=(\lambda_1,\lambda_2,\ldots,
	\lambda_m)\vDash \lvert{\mathcal{B}}\rvert$, and $T\in \pstab_\mathcal{B}
	(\lambda)$, then the number of tableaux in $\tab_\mathcal{B}(\lambda)$ that
	map to $T$ under $w_{\lambda,\mathcal{B}}$, $\lvert{w_{\lambda, \mathcal{B}}^{-1} \left(T\right)}\rvert$, is given by the following formula
	\begin{equation*}
	\lvert{w_{\lambda,\mathcal{B}}^{-1}\left(T\right)} \rvert=\prod_{i=0}^{m-1}
	\left(\lvert{\mathcal{B}}\rvert-\sum_{j=1}^{i}\lambda_j\right)\cdot
	\prod_{k=1}^{m} \left(\lambda_k-1\right)!
	\end{equation*}
\end{thm}
\begin{proof}
	The proof follows by induction on the number of columns of $T\in 
\pstab_\mathcal{B}\left(\lambda\right)$, for 
$\lambda=(\lambda_1,\ldots,\lambda_m)$.

	Case $m=1$. Considering the previous algorithm, any disposition of 
symbols from
	$\mathcal{B}$ in a column of length $\lambda_1=\lvert{\mathcal{B}}\rvert$
	will lead to the column tableau obtained by arranging these symbols in
	increasing order from bottom to top. There are $\lvert{\mathcal{B}}\rvert!$
	possibilities for tableaux in these circumstances and
	\begin{align*}
	\prod_{i=0}^{0}\left(\lvert{\mathcal{B}}\rvert -\sum_{j=1}^{0}\lambda_j\right)
	\cdot\prod_{k=1}^{1}\left(\lambda_k-1\right)!&=\lvert{\mathcal{B}}\rvert\cdot
	(\lambda_1-1)!\\
	&=\lvert{\mathcal{B}}\rvert\cdot(\lvert{\mathcal{B}}
	\rvert-1)!=\lvert{\mathcal{B}}\rvert!
	\end{align*}

	Fix $m>1$ and suppose by induction hypothesis that the result is true for
	$m-1$. That is, suppose that for all $\mathcal{C}\subseteq \aA$, $\lambda'=(
	\lambda_1',\ldots, \lambda_{m-1}')$ with $\lambda'\vDash\lvert\mathcal{C}\rvert$
	and $T'=c_1'c_2'\cdots c_{m-1}'\in \pstab_\mathcal{C}(\lambda')$,
	\begin{equation*}
	\lvert{w_{\lambda',\mathcal{C}}^{-1}\left(T'\right)}\rvert=
	\prod_{i=0}^{m-2}\bigg(\lvert{\mathcal{C}}\rvert-\sum_{j=1}^{i}\lambda_j'
	\bigg)\cdot \prod_{k=1}^{m-1}\left(\lambda_k'-1\right)!
	\end{equation*}
	Consider $\mathcal{B}\subseteq \aA_n$ for some $n\in\mathbb{N}$ and $\lambda=(\lambda_1,\lambda_2,\ldots,
	\lambda_m)\in \mathbb{N}^m$ such that $\lambda\vDash \lvert{\mathcal{B}}\rvert$.
	Given a \ps\ tableau $T=c_1c_2\cdots c_m \in \pstab_\mathcal{B}(\lambda)$,
	considering Algorithm~\ref{alg:pre_to_ps}, the tableaux that map to $T$
	under $w_{\lambda,\mathcal{B}}$ are the tableaux
	$\bar{T}=\bar{c_1}\bar{c_2}\cdots\bar{c_m}\in \tab_\mathcal{B}(\lambda)$ that
	contain the smallest symbol from $\mathcal{B}$ in any column-row position and
	such that, if $\bar{T}'=\bar{c_1}'\bar{c_2}'\cdots\bar{c_m}'\in \tab_\mathcal{B}
	(\lambda)$ is the tableau obtained from exchanging the positions of the
	smallest symbol of $\bar{T}$ with the symbol in the column-row position $(1,1)$
	of $\bar{T}$
	then $\cont(\bar{c_1}')= \cont(c_1)$ and $w_{(\lambda_2,\ldots,\lambda_m),
	\cont(\bar{c_2}'\cdots\bar{c_m}')} (\bar{c_2}'\cdots\bar{c_m}')=c_2\cdots c_m$.
	Since there are $\lvert{\mathcal{B}} \rvert$ column-row
	positions in $\bar{T}$ for the smallest symbol, and there are $\lambda_1 -1$
	positions in $\bar{T}'$ to be filled with $\lvert{\cont(c_1)}\rvert-1=\lambda_1-1$
	symbols (the smallest symbol from $\mathcal{B}$ is already in the bottom left box
	of $\bar{T}'$), there are
	\[\lvert{\mathcal{B}}\rvert\cdot(\lambda_1-1)!\cdot\lvert{w^{-1}_{(\lambda_2, \ldots,\lambda_m),
			\cont(c_2\cdots c_m)}(c_2\cdots c_m)}\rvert\]
	tableaux that will map to $T$ under $w_{\lambda,\mathcal{B}}$.
	By the induction hypothesis,
	\begin{align*}
	&\lvert{w^{-1}_{(\lambda_2,\ldots,\lambda_m),\cont(c_2\cdots c_m)}
		(c_2\cdots c_m)}\rvert\\
	={}&\prod_{i=1}^{m-1}\bigg((\lambda_2+\cdots+\lambda_m)-\sum_{j=2}^{i}
	\lambda_j\bigg)\cdot \prod_{k=2}^{m}\left(\lambda_k-1\right)!\\
	={}&\prod_{i=1}^{m-1}\bigg((\lvert{\mathcal{B}}\rvert-\lambda_1)-\sum_{j=2}^{i}
	\lambda_j\bigg)\cdot\prod_{k=2}^{m}\left(\lambda_k-1\right)!\ .
	\end{align*}
	So,
	\begin{align*}
	\lvert{w^{-1}_{\lambda,\mathcal{B}}(T)}\rvert&=\lvert{\mathcal{B}}\rvert\cdot
	(\lambda_1-1)!\cdot	\prod_{i=1}^{m-1}\bigg((\lvert{\mathcal{B}}\rvert-\lambda_1)-
	\sum_{j=2}^{i}\lambda_j\bigg)\cdot \prod_{k=2}^{m}\left(\lambda_k-1\right)!\\
	&=\prod_{i=0}^{m-1}\bigg(\lvert{\mathcal{B}}\rvert-\sum_{j=1}^{i}\lambda_j\bigg)
	\cdot\prod_{k=1}^{m}\left(\lambda_k-1\right)!\ .
	\end{align*}
	The result follows by induction.
\end{proof}
\begin{cor}\label{pstablambda}
	If $\mathcal{B}\subseteq \aA_n$ for some $n\in\mathbb{N}$, and
	$\lambda=(\lambda_1,\lambda_2,\ldots,\lambda_m)\vDash\lvert{\mathcal{B}}\rvert$,
	\begin{align*}
	\lvert{\pstab_\mathcal{B}(\lambda)}\rvert=\frac{(\lvert{\mathcal{B}}\rvert
		-1)!}{\prod_{i=1}^{m-1} \Big(\lvert{\mathcal{B}}\rvert-\sum_{j=1}^{i}
		\lambda_j\Big) \cdot \prod_{k=1}^{m}\left(\lambda_k-1\right)!}
	\end{align*}
\end{cor}
\begin{proof}
	Given $\mathcal{B}\subseteq\aA_n$ for some $n\in\mathbb{N}$ and $\lambda=(\lambda_1,\lambda_2,\ldots,
	\lambda_m)\vDash \lvert{\mathcal{B}}\rvert$, \[\lvert{\tab_\mathcal{B}(\lambda)}
	\rvert= \lvert{\mathcal{B}}\rvert!\]
	This number is also given by the sum of the cardinality of the pre-images of the
	tableaux $T\in \pstab_\mathcal{B}(\lambda)$ under $w_{\lambda,\mathcal{B}}$, that
	is,
	\[\lvert{\tab_\mathcal{B}(\lambda)}\rvert=\sum_{T\in \pstab_\mathcal{B}(\lambda)}\lvert{w^{-1}_{\lambda,\mathcal{B}}(T)}\rvert\]
	By the previous theorem, the number $\lvert{w^{-1}_{\lambda,\mathcal{B}}(T)}
	\rvert$ depends only on $\mathcal{B}$ and $\lambda$ and not on the tableau $T$.
	Thus, using the previous theorem it follows that
	\[\lvert{\mathcal{B}}\rvert!=\lvert{\pstab_\mathcal{B}(\lambda)}\rvert\cdot
	\prod_{i=0}^{m-1}\bigg( \lvert{\mathcal{B}}\rvert-\sum_{j=1}^{i} \lambda_j
	\bigg)\cdot \prod_{k=1}^{m}\left(\lambda_k-1\right)!\]
	So,
	\begin{align*}
	\lvert{\pstab_\mathcal{B}(\lambda)}\rvert&=\frac{\lvert{\mathcal{B}}\rvert!}
	{\prod_{i=0}^{m-1}\Big(\lvert{\mathcal{B}}\rvert-\sum_{j=1}^{i}\lambda_j
		\Big)\cdot \prod_{k=1}^{m} \left(\lambda_k-1\right)!}\\
	&=\frac{(\lvert{\mathcal{B}}\rvert-1)!}{\prod_{i=1}^{m-1}\Big(\lvert{
			\mathcal{B}}\rvert-\sum_{j=1}^{i}
		\lambda_j\Big)\cdot \prod_{k=1}^{m}\left(\lambda_k-1\right)!}.
	\end{align*}
	and the result follows.
\end{proof}

The $n$-th Bell number $B_n$ is the number of partitions of a set with $n$
distinct elements.  As noted at the beginning of Section~\ref{sec:bellnumbers},
the partitions of a set $\mathcal{B}\subseteq \aA_k$ for some $k\in\mathbb{N}$
are in one-to-one correspondence
with the standard \ps\ tableaux from the set $\pstab_\mathcal{B}$.
Thus, since
\begin{equation*}
\pstab_\mathcal{B}=\bigcup_{\substack{\lambda\vDash \lvert{\mathcal{B}}\rvert}} \pstab_\mathcal{B}(\lambda),
\end{equation*}
if $n=\lvert{\mathcal{B}}\rvert$, the $n$-th Bell number is also given by
\[\lvert{\pstab_\mathcal{B}}\rvert=\sum_{\lambda\vDash n}\lvert{\pstab_\mathcal{B}(\lambda)}\rvert\]
and therefore
\begin{thm}
	For any $n\in \mathbb{N}$, the $n$-th Bell number, $B_n$, is given by the following formula
	\[B_n=\sum_{(\lambda_1,\ldots,\lambda_m)\vDash n}\Bigg(\frac{(n-1)!}{\prod_{i=1}^{m-1}\Big(n-\sum_{j=1}^{i}\lambda_j\Big)\cdot \prod_{k=1}^{m}\left(\lambda_k-1\right)!}\Bigg).\]
\end{thm}

For example, considering $\mathcal{B}=\{1,2,3,4\}$
\begin{align*}
B_4&=\lvert{\pstab_\mathcal{B}\left((4)\right)}\rvert+\lvert{\pstab_\mathcal{B}\left((3,1)\right)}\rvert+\lvert{\pstab_\mathcal{B}\left((1,3)\right)}\rvert\\
&+\lvert{\pstab_\mathcal{B}\left((2,2)\right)}\rvert+\lvert{\pstab_\mathcal{B}\left((2,1,1)\right)}\rvert+\lvert{\pstab_\mathcal{B}\left((1,2,1)\right)}\rvert\\
&+\lvert{\pstab_\mathcal{B}\left((1,1,2)\right)}\rvert+\lvert{\pstab_\mathcal{B}\left((1,1,1,1)\right)}\rvert\\
&=1+3+1+3+3+2+1+1=15.
\end{align*}

The algorithm of extended \ps\ insertion from Section~\ref{sec:rsklike} (Algorithm~\ref{alg:ExtendedPS}) only defines an injection between the set of standard words $\mathfrak{S}(\mathcal{B})$, where $\mathcal{B}\subseteq \aA_n$ for some $n\in\mathbb{N}$, and the set of pairs of standard \ps\ tableaux of the same shape over $\mathcal{B}$, $\bigcup_{\lambda\vDash \lvert{\mathcal{B}}\rvert}\left(\pstab_\mathcal{B}(\lambda)\times\pstab_\mathcal{B}(\lambda)\right)$. Therefore,

\begin{prop}
	For any $n\in\mathbb{N}$,
	\[n!\leq \sum_{(\lambda_1,\ldots,\lambda_m)\vDash n}\Bigg(\frac{(n-1)!}{\prod_{i=1}^{m-1}\Big(n-\sum_{j=1}^{i}\lambda_j\Big)\cdot \prod_{k=1}^{m}\left(\lambda_k-1\right)!}\Bigg)^2.\]
\end{prop}
\begin{proof}
	Follows from the observations before the proposition together with the fact that for any $\mathcal{B}\subseteq \aA_k$ for some $k\in\mathbb{N}$, with $\lvert{\mathcal{B}}\rvert=n$, $\lvert{\mathfrak{S}(\mathcal{B})}\rvert=n!$ and Corollary~\ref{pstablambda}.
\end{proof}

\begin{prop}
	For any $k,n\in\mathbb{N}$ and $\mathcal{B}\subseteq \aA_k$ such that $\lvert{\mathcal{B}}\rvert=n$, if $\lambda\vDash n$ and $T\in \pstab_\mathcal{B}(\lambda)$, then
	\[\left\lvert{\left\{\sigma \in \mathfrak{S}(\mathcal{B}):\pP(\sigma)=T\right\}}\right\rvert\leq \frac{(n-1)!}{\prod_{i=1}^{m-1}\Big(n-\sum_{j=1}^{i}\lambda_j\Big)\cdot \prod_{k=1}^{m}\left(\lambda_k-1\right)!}.\]
\end{prop}
\begin{proof}
	It is straightforward that
	\[\left\lvert{\left\{\sigma \in \mathfrak{S}(\mathcal{B}):\pP(\sigma)=T\right\}}\right\rvert=\left\lvert{\left\{\big(\pP(\sigma),\qQ(\sigma)\big):\sigma\in\mathfrak{S}(\mathcal{B})\wedge \pP(\sigma)=T\right\}}\right\rvert.\]
	It follows that
	\[\left\{\left(\pP(\sigma),\qQ(\sigma)\right):\sigma\in \mathfrak{S}(\mathcal{B})\wedge \pP(\sigma)=T\right\}\subseteq \left\{T\right\}\times \pstab_\mathcal{B}(\lambda)\]
	and thus we deduce the result.
\end{proof}

The inequality of the previous proposition is, in general, strict. For instance,
considering the set $\pstab_{\{2,4,5,6\}}\left((2,2)\right)=\left\{\ytableausetup
{aligntableaux=center}\begin{ytableau}
4 & 6\\
2 & 5
\end{ytableau},\begin{ytableau}
5 & 6\\
2 & 4
\end{ytableau},\begin{ytableau}
6 & 5\\
2 & 4
\end{ytableau}\right\}$ and fixing $T=\begin{ytableau}
5 & 6\\
2 & 4
\end{ytableau}$, then reading $T$ according to all the possibilities for recording
tableaux in $\pstab_{\{2,4,5,6\}}\left((2,2)\right)$ leads to the standard
words $\sigma=5264, \tau=5624, \upsilon=5642$, respectively. However, according to
Algorithm~\ref{alg:PSinsertion}, only $\sigma$ and $\tau$ insert to $T$ .


\bibliographystyle{alphaabbrv}
\bibliography{mybib}

\end{document}